\documentclass[12pt]{amsart}
\usepackage{amsmath}
\usepackage{amsthm}
\usepackage{amssymb}
\usepackage{comment}

\newtheorem{theorem}{Theorem}[section]
\newtheorem{lemma}[theorem]{Lemma}
\newtheorem{proposition}[theorem]{Proposition}
\newtheorem{corollary}[theorem]{Corollary}

\newtheorem{remar}[theorem]{Remark}

\newenvironment{remark}{\begin{remar}\rm}{\gzaun\end{remar}}
\newcommand{\gzaun}{\unskip\nobreak\hfil\penalty50%
\hskip1em\hbox{}\nobreak\hfil%
$\#$\parfillskip=0pt\finalhyphendemerits=0}

\newcommand{\bfind}[1]{\index{#1}{\bf #1}}
\newcommand{\n}{\par\noindent}

\newcommand{\sn}{\par\smallskip\noindent}
\newcommand{\mn}{\par\medskip\noindent}
\newcommand{\bn}{\par\bigskip\noindent}
\newcommand{\pars}{\par\smallskip}
\newcommand{\parm}{\par\medskip}

\newcommand{\cal}{\mathcal}
\newcommand{\Aut}{\mbox{\rm Aut}\,}
\newcommand{\sep}{^{\rm sep}}

\newcommand{\chara}{\mbox{\rm char}\,}

\newcommand{\Gal}{\mbox{\rm Gal}\,}

\newcommand{\smin}{_{\rm min}}

\newcommand{\rme}{\mbox{\rm e}\,}
\newcommand{\rmf}{\mbox{\rm f}\,}

\newcommand{\rmid}{\mbox{\rm id}}
\newcommand{\trmid}{\mbox{\tiny\rm id}}
\newcommand{\smid}{\setminus\{\rmid\}}

\newcommand{\cO}{\mathcal{O}}
\newcommand{\cM}{\mathcal{M}}

\newcommand{\cE}{\mathcal{E}}

\newcommand{\Z}{\mathbb Z}
\newcommand{\N}{\mathbb N}
\newcommand{\Q}{\mathbb Q}
\newcommand{\R}{\mathbb R}
\newcommand{\F}{\mathbb F}
\begin{document}

\title[Topics in higher ramification theory I]{Topics in higher ramification theory I:
ramification ideals}
\author{Franz-Viktor Kuhlmann}
\address{Institute of Mathematics, University of Szczecin,
ul.\ Wielkopolska 15, 70-451 Szczecin, Poland}
\email{fvk@usz.edu.pl}

\thanks{The author would like to thank Josnei Novacoski for helpful discussions and
for pointing out a mistake in the example given in Section~\ref{sect2ex}, as well as
Anna Rzepka for delivering a main ingredient for this example. He also thanks
Konstantinos Kartas for helping to develop the last example in Section~\ref{sectIeqM}.
Heartfelt thanks go to the referee for carefully reading the manuscript and
finding many mistakes and typos.}



\date{3.\ 11.\ 2026}

\begin{abstract}
We introduce and study the notion of ramification ideals in higher ramification theory.
After general results on their computation for finite extensions, we discuss their
connection with the possibly nontrivial defect of the extensions. We
compute them for Artin-Schreier extensions and Kummer extensions of prime degree
equal to the residue characteristic, which may or may not have nontrivial defect. We
present an example that shows that nontrivial defect in an extension of degree $p^2$,
$p$ a prime, may not imply the existence of a nonprincipal ramification ideal.
\end{abstract}

\maketitle

\begin{center}
{\rm Dedicated to Dr.\ Sudesh K.\ Khanduja\\
on the occasion of her 75th birthday}
\end{center}

\mn
%
%
\section{Introduction}
Higher ramification theory is the theory of valued field extensions $\cE=(L|K,v)$
where $(K,v)$ has positive residue characteristic $p$ and is its own \bfind{absolute
rami\-fication field} (see Section~\ref{sectrf}). The latter means that $(K,v)$ is
henselian, its \bfind{value group} $vK$ is divisible by all primes different from $p$,
and its \bfind{residue field} $Kv$ is separable-algebraically closed. The
\bfind{absolute Galois group} $\Gal K:=\Gal K\sep|K$, where $K\sep$ denotes the
separable-algebraic closure of $K$, is then a pro-$p$-group
(\cite[(20.18), p.\ 167]{En}. This implies that every
finite Galois extension of $K$ is a tower of Galois extensions of degree $p$. In
\bfind{equal characteristic}, i.e., if $\chara K=\chara Kv=p$, the latter are
\bfind{Artin-Schreier extensions}, and in \bfind{mixed characteristic}, i.e., if
$\chara K=0$ and $\chara Kv=p$, they are \bfind{Kummer extensions} because $K$
contains all $p$-th roots of unity (see Section~\ref{sectGaldefdegp}).

\pars
Our interest in higher ramification theory originated from the following well
known deep open valuation theoretical problems:
\sn
1) local uniformization, the local form of resolution of singularities in arbitrary dimension and positive characteristic,
\sn
2) decidability of the field $\F_q((t))$ of Laurent series over a finite field
$\F_q$ with $q$ elements, and of its perfect hull, where $q$ is a power of a prime $p$.
\sn
Both problems are connected with the structure theory of valued function fields of
positive characteristic $p$. The main obstruction here is the phenomenon of the
\bfind{defect}, which we define in Section~\ref{sectdef}. For background on the
defect and its impact on the above problems, see \cite{Ku31} for generalities on the
defect and a large collection of examples for defect extensions, \cite{Ku11,Ku23,Ku26}
for partial solutions of the local uniformization problem and the connection with the
defect, \cite{Ku39,Ku45} for model theoretic results on tame and separably tame fields,
and \cite{Ku29,Ku55} for elimination of ramification (and hence of the influence of the
defect) in valued function fields, which provides the crucial tools for the proofs of
the main results in \cite{Ku23,Ku26,Ku39,Ku45}.

An important approach in the study of the defect in extensions of arbitrary finite degree
is what Cossart and Piltant use in \cite[\S8]{CP} and call ``Abhyankar's strategy'';
it was introduced by Abhyankar in \cite{A}. Using the above mentioned property of the
absolute ramification field, it reduces the investigation
of a given finite extension of valued fields to that of purely inseparable extensions and
of Galois extensions of degree $p=\chara Kv>0$; if the original extension has nontrivial
defect, then the same will be true for some of the latter extensions (cf.\
\cite[Section~2.1]{pr1}). This principle has been crucially used in \cite{Ku29,Ku55}.

Defects of Galois extensions $\cE=(L|K,v)$ of prime degree
have been classified (``dependent'' vs.\ ``independent'' defect) first in \cite{Ku30}
for the equal characteristic case and then in \cite{KuRz} in general. Theorem~1.4 of
\cite{KuRz} presents various criteria for independent defect. One of them uses the
ramification ideal $I_\cE$ of the extension $\cE$, which we define in
Section~\ref{secthrg}.

Section~\ref{sectcompri} is then devoted to the computation of ramification ideals.
Starting in Section~\ref{sectann} with a first approach described by Ribenboim in
\cite{R2} we develop more
elaborate computations. Of particular interest is the case of extensions that have
valuation bases; for this notion, see Section~\ref{sectvb}. Based on this,
Proposition~\ref{compri} treats towers of two Galois extensions where the upper one
has a valuation basis, which we will apply in Section~\ref{sect2ex}. We do not know
an analogue of Proposition~\ref{compri} where the lower extension has a valuation basis.

In Section~\ref{sectridef} we discuss the relation between defect and the existence
of nonprincipal ramification ideals. While it is true that a finite Galois extension
without defect has only principal ramification ideals (see Proposition~\ref{dlrip}),
the converse does not hold. We give an example for this phenomenon in
Section~\ref{sect2ex} for the equal characteristic case. We do not know whether
there is a similar example in mixed characteristic. The example consists of a tower of
two Galois extensions of prime degree where the lower has defect but the upper does not.
We do not know of an example where this is reversed.

In Section~\ref{sectGaldefdegp} we first compute the unique ramification ideals
$I_\cE$ for Artin-Schreier and Kummer extensions $\cE=(L|K,v)$ of degree $p=\chara Kv$
without defect; the results are applied in \cite{pr2}. We then take a closer look at the
ramification ideals $I_\cE$ in the case when $\cE$ is a defect extension, which are
computed in \cite{pr1}.

In Section~\ref{sectIeqM} we determine the cases where $I_\cE$ is equal to the maximal
ideal $\cM_L$ of the valuation ring on $L$. We discuss the significance of this equality
and give examples of several cases where it holds. Already in Proposition~\ref{ASKgen}
a dichotomy can be observed when it comes to the nature of generators for Kummer
extensions without defect; it appears again in the computation of their ramification
ideals (see Theorem~\ref{ThmASKgen}). The examples we present for the cases of Kummer
extensions illuminate the dichotomy.

\pars
The following theorem shows the particular importance of extensions of prime degree of
a valued field $(K,v)$. For the definition of the notions of {\bf deeply ramified field}
and of {\bf module of relative differentials}, see \cite{Kudefc}. Denote the valuation
ring of $(K,v)$ by $\cO_K\,$.
\begin{theorem}{\cite[Theorem 1.3]{pr2}}                     \label{GRthm+}
Let $(K,v)$ be a valued field of residue characteristic $p>0$. If $K$ has characteristic
$0$, then assume in addition that it contains all $p$-th roots of unity. Then $(K,v)$ is
a deeply ramified field if and only if the module $\Omega_{\cO_L|\cO_K}$ of relative
differentials equals zero for all unibranched Galois extensions $(L|K,v)$ of
degree $p$.
\end{theorem}
\n
This theorem as well as Abhyankar's strategy justify our particular focus on
Artin-Schreier extensions and Kummer extensions of prime degree in the present paper.
\pars
Several of the results of this paper are used in \cite{Kudefc}, where we present
$\Omega_{\cO_L|\cO_K}$ using $\cO_L$-ideals that are first order definable in the
language of valued fields with a predicate for membership in $K$. These results are:
\n
1) Proposition~\ref{ris} which establishes an inclusion preserving bijection from the set
of all higher ramification groups of a valued Galois extension to the set of all higher ramification ideals;
\n
2) Proposition~\ref{rj} which shows that the definition of ramification ideals used in
\cite{Kudefc} is equivalent to the definition given in \cite{KuRz} by use of ramification jumps;
\n
3) Proposition~\ref{dlrip} which shows that every ramification ideal of a
unibranched Galois extension without defect is principal;
\n
4) Theorem~\ref{ThmASKgen} which presents the ramification ideals of
unibranched Galois extensions of prime degree without defect;
\n
5) part 1) of Corollary~\ref{ASdef} which gives a criterion for a valued Artin-Schreier
extension to have nontrivial defect without the assumption that the extension is
unibranched (we have not found such a criterion in the literature).

\pars
Let us mention that some approaches and results of the present paper are generalized in
\cite{Nov2}. In particular, Theorem~\ref{ThmASKgen} and Proposition~\ref{ThmIEdef} can
be directly deduced from Proposition~\ref{ASKgen} by using \cite[Theorem 3.1]{Nov2}.

%
%
\section{Preliminaries}
For basic facts from valuation theory, see \cite{En,EP,R0,W,ZS2}.
\pars
Take a valued field $(K,v)$. We denote its value group by $vK$, its residue field by
$Kv$, its valuation ring by $\cO_K\,$, with its maximal ideal $\cM_K\,$.
For $a\in K$, we write $va$ for its value and $av$ for its residue. {\it We will tacitly
assume all valuations appearing in this paper to be nontrivial.}

By $(L|K,v)$ we denote an extension $L|K$ with a valuation $v$ on $L$, where $K$
is endowed with the restriction of $v$. In this case, there are induced embeddings of
$vK$ in $vL$ and of $Kv$ in $Lv$. The extension $(L|K,v)$ is called \bfind{immediate}
if these embeddings are onto.
\pars
By $\tilde{K}$ we denote the algebraic closure of $K$. Every valuation on $K\sep$ has
a unique extension to $\tilde{K}$; this allows us to identify the absolute Galois group
$\Gal K=\Gal K\sep|K$ with the automorphism group $\Aut\tilde{K}|K$.

%
%
\subsection{Final segments and ideals}           \label{sectfins}
A subset $S$ of a totally ordered set $\Gamma$ is an \bfind{initial segment} of
$\Gamma$ if for every $\alpha\in S$ and every $\gamma\in\Gamma$ with $\gamma \leq
\alpha$, it follows that $\gamma\in S$. Symmetrically, $S$ is a \bfind{final segment}
of $\Gamma$ if for every $\alpha \in S$ and every $\gamma\in
\Gamma$ with $\gamma\geq \alpha$, it follows that $\gamma\in S$. {\it Throughout,
we will assume initial and final segments to be nonempty and not equal to $\Gamma$.}

\pars
Take a valued field $(L,v)$. A subset $I\subset L$ is a \bfind{fractional
${\cal O}_L$-ideal} if there is some $a\in L$ such that $aI$ is an ${\cal O}_L$-ideal
(contained in ${\cal O}_L$). In particular, every fractional ${\cal O}_L$-ideal is a
proper subset of $L$.

The function
\begin{equation}                           \label{vIS}
v:\; I\>\mapsto\> \Sigma_I\>:=\> \{vb\mid 0\ne b\in I\}
\end{equation}
is a bijection from the set of all nonzero fractional ${\cal O}_L$-ideals onto the set
of all final segments of $vL$. Both sets are linearly ordered by
inclusion, and the function (\ref{vIS}) is order preserving: $J\subseteq I$ holds
if and only if $\Sigma_J\subseteq\Sigma_I$ holds. The inverse of the above function
is the order preserving function
\begin{equation}                          \label{vSI}
\Sigma\>\mapsto\> I_\Sigma\>:=\>\{a\in L\mid va\in \Sigma\}\cup\{0\}\>=\>
(a\in L\mid va\in \Sigma)\>
\end{equation}
from the set of all final segments of $vL$ onto the set of all nonzero fractional
${\cal O}_L$-ideals.

%
%
\subsection{The defect}           \label{sectdef}
\mbox{ }\sn
A valued field extension $(L|K,v)$ is \bfind{unibranched} if the
extension of $v$ from $K$ to $L$ is unique. Note that a unibranched extension is
automatically algebraic, since every transcendental extension always admits several
extensions of the valuation. A valued field $(K,v)$ is \bfind{henselian} if it
satisfies Hensel's Lemma, or equivalently, if all of its algebraic extensions are
unibranched. A \bfind{henselization} of $(K,v)$ is an algebraic
extension of $(K,v)$ which admits a valuation preserving embedding in every other
henselian extension of $(K,v)$. Henselizations always exist and are unique up to
valuation preserving isomorphism over $K$; therefore we will talk of {\it the}
henselization of $(K,v)$ and denote it by $(K,v)^h=(K^h,v)$. The henselization of
$(K,v)$ is an immediate separable-algebraic extension. The valued field $(K,v)$ is
henselian if and only if it is equal to its henselization.

If $(L|K,v)$ is a finite unibranched extension, then by the Lemma of Ostrowski
(\cite[Corollary to Theorem~25, Section G, p.~78]{ZS2}),
\begin{equation}                    \label{feuniq}
[L:K]\>=\> \tilde{p}^{\nu }\cdot(vL:vK)[Lv:Kv]\>,
\end{equation}
where $\nu$ is a non-negative integer and $\tilde{p}$ the
\bfind{characteristic exponent} of $Kv$, that is, $\tilde{p}=\chara Kv$ if it is 
positive and $\tilde{p}=1$ otherwise. The factor $d(L|K,v):=\tilde{p}^{\nu }$ is 
the \bfind{defect} of the extension $(L|K,v)$. We call $(L|K,v)$ a \bfind{defect
extension} if $d(L|K,v) >1$, and a \bfind{defectless extension} if $d(L|K,v)=1$.
Throughout this paper, when we talk of a \bfind{defect extension $(L|K,v)$ of prime
degree}, we will always tacitly assume that it is a unibranched extension. If the
extension is of prime degee, then it follows from (\ref{feuniq}) that $[L:K]=\chara Kv$
and that $(vL:vK)=1=[Lv:Kv]$, that is, $(L|K,v)$ is an immediate extension.

Nontrivial defect only appears when $\chara Kv=p>0$, in which case $\tilde{p}=p$.
A henselian field $(K,v)$ is called a \bfind{defectless field} if all of its finite
extensions are defectless.

The following lemma shows that the defect is multiplicative.
This is a consequence of the multiplicativity of the degree of field
extensions and of ramification index and inertia degree. We leave the
straightforward proof to the reader.
\begin{lemma}                                       \label{md}
Take a valued field $(K,v)$. If $L|K$ and $M|L$ are finite extensions
and the extension of $v$ from $K$ to $M$ is unique, then
\begin{equation}         \label{pf}
\mbox{\rm d}(M|K,v) = \mbox{\rm d}(M|L,v)\cdot\mbox{\rm d}(L|K,v)
\end{equation}
In particular, $(M|K,v)$ is defectless if and only if $(M|L,v)$ and
$(L|K,v)$ are defectless.
\end{lemma}

For every extension $(L|K,v)$ of valued fields and $a\in L$ we define
\[
v(a-K)\>:=\>\{v(a-c)\, |\, c\in K\}\>.
\]
The set $v(a-K)\cap vK$ is an initial segment of $vK$. For more information on its
properties, see \cite{Ku65}.

\begin{lemma}                         \label{K(a)Kh(a)}
Take a unibranched algebraic extension $(K(a)|K,v)$ and an extension of $v$ from $K(a)$
to $\tilde{K}$. Denote by $(K^h,v)$ the henselization of $(K,v)$ in $(\tilde{K},v)$.
Then:
\sn
a) $K(a)|K$ is linearly disjoint from $K^h|K$,
\sn
b) $(K^h(a)|K^h,v)$ is a defect extension if and only if $(K(a)|K,v)$ is, and
\sn
c) $v(a-K^h)=v(a-K)$.
\end{lemma}
\begin{proof}
Our first assertion follows from \cite[Lemma~2.1]{[BK2]}. For the proof of the second
assertion, recall that henselizations are immediate extensions, so we have $vK^h=vK$
and $K^hv=Kv$. Further, we have $K^h(a)=K(a)^h$ since on the one hand, $K^h(a)$ is
henselian, being an algebraic extension of $K^h$, and on the other hand, it contains
$K(a)$. Hence, $vK^h(a)=vK(a)$ and
$K^h(a)v=K(a)v$. Since $K(a)|K$ is linearly disjoint from $K^h|K$, we also have
$[K^h(a):K^h]=[K(a):K]$. As an algebraic extension of a henselian field, $(K^h(a)|K^h,
v)$ is unibranched. It follows that
\begin{eqnarray*}
d(K^h(a)|K^h,v)&=&[K^h(a):K^h]/(vK^h(a):vK^h)[K^h(a)v:K^hv]\\
&=& [K(a):K]/(vK(a):vK)[K(a)v:Kv] \\
&=& d(K(a)|K,v)\>.
\end{eqnarray*}
This proves our second assertion.

Suppose that $v(a-K^h)\ne v(a-K)$. Since $v(a-K)$ is an initial segment of $vK=vK^h$,
this means that there must be some $z\in K^h$ such that $v(a-z)>v(a-K)$. However, as
$K(a)|K$ is linearly disjoint from $K^h|K$, we know from  that
this cannot be true. This proves our third assertion.
\end{proof}

%
%
\subsection{The ramification field}          \label{sectrf}
\mbox{ }\sn
The \bfind{ramification field} of a Galois extension $(L|K,v)$ with Galois group
$G=\Gal(L|K)$ is the fixed field in $L$ of the \bfind{ramification group}
\begin{equation}
G^r\>:=\> \left\{\sigma\in G\>\left|\;\> \frac{\sigma b -b}{b}\in \cM_L
\mbox{ \ for all }b\in L^\times\right.\right\}\>.
\end{equation}
In order to reduce the study of arbitrary finite defect extensions to purely
inseparable extensions
and Galois extensions of degree $p=\chara Kv>0$, we fix an extension of $v$ from $K$ 
to $\tilde{K}$. The \bfind{absolute ramification field of
$(K,v)$} (with respect to the chosen extension of $v$), denoted by $(K^r,v)$, is the
ramification field of the Galois extension $(K\sep|K,v)$.

Take a finite defect extension $(L|K,v)$. Then $(L.K^r|K^r,v)$ is a
defect extension with the same defect (see \cite[Proposition~2.12]{KuRz}). On the
other hand, $K\sep|K^r$ is a $p$-extension (see \cite[Lemma 2.7]{Ku30}), so $K^r(a)|
K^r$ is a tower of purely inseparable extensions and Galois extensions of degree $p$.
Note that 
$(K,v)=(K^r,v)$ if and only if $(K,v)$ is henselian, $vK$ is divisible by
all primes different from $\chara Kv$, and $Kv$ is separable-algebraically closed;
this follows from \cite[Corollary 5.3.8]{EP}.

%
%
\subsection{Immediate extensions}         
\mbox{ }\sn
Let us give more details about immediate extensions.
\begin{lemma}                              \label{immb-c}
Take an arbitrary extension $(L|K,v)$ and $b\in L$. Then there is $c\in K$ such that 
$v(b-c)>vb$ if and only if $vb\in vK$ and $c'bv\in Kv$ for every $c'\in K$ such that 
$vc'b=0$. 
\end{lemma}
\begin{proof}
Assume first that $v(b-c)>vb$. Then $vb=vc\in vK$ and for any $c'\in K$ such that 
$vc'b=0$ we have $v(c'b-c'c)>0$ so that $c'bv=c'cv\in Kv$. Now assume that $vb\in vK$ and 
$c'bv\in Kv$ for every $c'\in K$ such that $vc'b=0$. Take $c_1\in K$ such that $vc_1=vb$
and set $c'=c_1^{-1}$. Then $vc'b=0$, hence by assumption, $c'bv\in Kv$. Take $c_2\in 
K$ such that $c'bv=c_2v$, so that $v(c'b-c_2)>0$. Multiplying with $c_1$ we obtain
$v(b-c_1c_2)>vc_1=vb$.
\end{proof}
It follows that an extension $(L|K,v)$ is immediate if and only if for all
$b\in L$ there is $c\in K$ such that $v(b-c)>vb$. This lays the basis for the proof of 
the following theorem; see \cite[Theorem 1]{K} and \cite[Lemma 2.29]{Ku65}.
\begin{theorem}                                              \label{CutImm}
If $(L|K,v)$ is an immediate extension of valued fields, then for every element 
$a\in L\setminus K$ the set $v(a-K)$ is an initial segment of $vK$ without maximal 
element.
\end{theorem}

The following partial converse of this theorem also holds, as shown in
\cite[Lemma 2.21]{Ku30}:
\begin{lemma}                                   \label{imm_deg_p}
Assume that $(K(a)|K,v)$ is a unibranched extension of prime degree such 
that $v(a-K)$ has no maximal element. Then the
extension $(K(a)|K,v)$ is immediate and hence a defect extension.
\end{lemma}

\mn
%
%
\subsection{Higher ramification groups and ramification ideals}       \label{secthrg}
\mbox{ }\sn
Take a valued field extension $\cE=(L|K,v)$. Assume that $L|K$ is a Galois extension,
and let $G=\Gal L|K$ denote its Galois group. We define the \bfind{series of upper
ramification groups}
\begin{equation}
G_I\>:=\>\left\{\sigma\in G\>\left|\;\> \frac{\sigma b -b}{b}\in I
\mbox{ \ for all }b\in L^\times\right.\right\}\>,
\end{equation}
where $I$ runs through all ${\cal O}_L$-ideals (cf.\ \cite[\S12]{ZS2}). Every
$G_I$ is a normal subgroup of $G$ (\cite[(d) on p.79]{ZS2}). Note that $G_{\cO_L}$
is the decomposition group of $\cE$; this easily follows from its definition as
$\{\sigma\in G\mid v\sigma=v\mbox{ on } L\}$ (see \cite[\S5 (16)]{Ku11}). This also
implies that $\cE$ is unibranched if and only if $G=G_{\cO_L}$ (cf.\
\cite[Theorem 5.4 b)]{Ku11}).

Further, $G_{\cM_L}$ is the ramification group of $\cE$ (see \cite[\S5 (18)]{Ku11}).
We call $G_I$ a \bfind{higher ramification group} if it is a nontrivial subgroup of
$G_{\cM_L}\,$. Every $G_I$ is a higher ramification group if and only if $G=
G_{\cM_L}$. If this is the case, then the extension $\cE$ is \bfind{purely wild} in
the sense of \cite{Ku2}, that is, $(vL:vK)$ is a $p$-power, where $p=\chara Kv$, and
$Lv|Kv$ is purely inseparable. Indeed, observe first that $G=G_{\cM_L}$ means that
$K$ is the ramification field of $(L|K,v)$, which by \cite[(20.15), p.\ 166]{En} is
equal to $K^r\cap L$. Now $K^r\cap L=K$ implies that $L|K$ and $K^r|K$ are linearly,
which by \cite[Lemma 4.2]{Ku2} is equivalent to $(L|K,v)$ being purely wild. Note that
every purely wild extension is unibranched since $K^r$ contains the henselization of
$(K,v)$ (under the chosen extension of $v$ to the algebraic closure), and by \cite[Theorem~2]{Ku34}, $L$ being linearly disjoint from the henselization over $K$
implies that $(L|K,v)$ is unibranched.

\pars
The function
\begin{equation}                            \label{eq12}
\varphi:\; I\>\mapsto\>G_I
\end{equation}
from the set of ${\cal O}_L$-ideals to the set of upper ramification groups
preserves $\subseteq$, that is, if $I\subseteq J$, then $G_I\subseteq G_J\,$.
As ${\cal O}_L$ is a valuation ring, the set of its ideals is linearly
ordered by inclusion. This shows that also the series of upper ramification
groups is linearly ordered by inclusion. Note that in general, $\varphi$ will neither
be injective nor surjective as a function to the set of normal subgroups of~$G$. This
gives rise to the task to determine the smallest ideal that is sent by $\varphi$ to
a group $G_I$ in its image. To this end, for each subgroup $H$ of $G$ we
define the (possibly fractional) $\cO_L$-ideal
\begin{equation}                            \label{genI}
I_H\>:=\>\left(\frac{\sigma b -b}{b}\>\left|\;\> \sigma\in H\,,\>b\in L^\times\right.
\right)\>=\>\left(\frac{\sigma b}{b}-1\>\left|\;\> \sigma\in H\,,\>b\in L^\times\right.
\right)\>
\end{equation}
and consider the function
\begin{equation}                            \label{eq13}
\psi:\; H\>\mapsto\>I_H
\end{equation}
from the set of all subgroups $H$ of $G_{\cO_L}$ to the set of all
${\cal O}_L$-ideals. Also
$\psi$ preserves $\subseteq$ and is in general neither injective
nor surjective. However, it is easy to see that $G_{(0)}=\{\rmid\}$ and
$I_{\{\trmid\}}=(0)$. If $I_H$ is nonzero, then it is called it a \bfind{ramification
ideal}; if in addition it is contained in $\cM_L\,$, then it is called it a
\bfind{higher ramification ideal}. Observe that the only ramification ideal that
is not a higher ramification ideal is $\cO_L$ and therefore, every ramification ideal
of a purely wild extension is a higher ramification ideal; so when we deal with such
extensions, we will drop the word ``higher'' for simplicity.

\pars
We note:
\begin{proposition}                      \label{ris}
1) For every ${\cal O}_L$-ideal $J$, the group $G_J$ is the largest of
all subgroups $H'$ of $G$ such that $I_{H'}\subseteq J$.
\sn
2) For every subgroup $H$ of $G$, the ideal $I_H$ is the smallest of all
$\cO_L$-ideals $J'$ such that $H\subseteq G_{J'}\,$.
\sn
3) For every subgroup $H$ of $G$ and every $\cO_L$-ideal $J$,
\[
I_{G_{I_H}}\>=\> I_H \;\mbox{ and }\; G_{I_{G_J}}\>=\> G_J\>.
\]
hence $\varphi$ is an inclusion preserving bijection from the set of all ramification
ideals together with $(0)$ onto the set of all upper ramification groups, with $\psi$
its inverse.

\sn
4) The function $\varphi$ induces an inclusion preserving bijection from the set of
all higher ramification ideals onto the set of all higher ramification groups,
with its inverse induced by $\psi$.
%
%

\sn
5) A subgroup $H$ of $G$ is an upper ramification group if and only if it is a
subgroup of $G_{\cO_L}$ and for every subgroup $H'$ of $G_{\cO_L}$ we have
$H\subsetneq H'\Rightarrow I_H\subsetneq I_{H'}\,$.

\sn
6) An $\cO_L$-ideal $J$ is a ramification ideal if and only if it is nonzero and for
every $\cO_L$-ideal $J'$ we have $J'\subsetneq J\Rightarrow G_{J'}\subsetneq G_J\,$.

\sn
7) If $\cE=(L|K,v)$ is a nontrivial unibranched Galois extension, then $I_G$ is its
largest ramification ideal. If in addition $\cE$ is of prime degree, then $I_G$ is
its unique ramification ideal.
\end{proposition}
\begin{proof}
1) and 2) follow directly from the definitions of $G_J$ and $I_H\,$.
\sn
3): If $J=I_H$, then it follows from part 1) that $H\subseteq G_J\,$.
Thus $J=I_H\subseteq I_{G_J}\subseteq J$, so $I_{G_J}=J$.
If $H=G_J$, then it follows from part 2) that $I_H\subseteq J$. Thus
$H\subseteq G_{I_H}\subseteq G_J=H\,$, so $G_{I_H}=H$.

\sn
4): If $I_H$ is a higher ramification ideal, then $I_H$ is nonzero and contained in
$\cM_L\,$, whence $G_{I_H}\subseteq G_{\cM_L}\,$, so $G_{I_H}$ is a higher ramification
group. This shows that $\varphi$ sends higher ramification ideals to higher
ramification groups.

If $G_J$ is a higher ramification group, then $G_J\subseteq G_{\cM_L}$, hence
$I_{G_J}\subseteq I_{G_{\cM_L}}\subseteq\cM_L\,$, where the second inequality follows from part 1). Moreover, since $G_J$ is nontrivial, $I_{G_J}$ is nonzero. This shows
that $\psi$ sends higher ramification groups to higher ramification ideals.

Now the assertion of part 4) follows from part 3).

\sn
5): Assume first that $H$ is an upper ramification group. Then it is a subgroup of
$G_{\cO_L}$ and there is an $\cO_L$-ideal $J$ such that $H=G_J\,$. Take a subgroup
$H'$ of $G_{\cO_L}$ which properly contains $G_J\,$. Then by part 1), $I_H=I_{G_J}\subsetneq I_{H'}\,$.

For the converse, assume that $H$ is a subgroup of $G_{\cO_L}$ such that for every
subgroup $H'$ of $G_{\cO_L}$ we have $H\subsetneq H'\Rightarrow I_H\subsetneq I_{H'}\,$.
By part 1), $G_{I_H}$ is the largest of all subgroups $H'$ of $G_{\cO_L}$ such that
$I_{H'}\subseteq I_H\,$. Therefore $G_{I_H}=H$, which shows that $H$ is an upper
ramification group.

\sn
6): Assume first that $J$ is a ramification ideal. Then it is nonzero and by part 3)
there is an upper ramification group $H$ such that $J=I_H\,$. If $J'$ is an
$\cO_L$-ideal properly contained in $I_H\,$, then by parts 2) and 3), $G_{J'}
\subsetneq H=G_{I_H}=G_J\,$.

For the converse, assume that $J$ is a nonzero $\cO_L$-ideal such that for every
$\cO_L$-ideal $J'$ we have $J'\subsetneq J\Rightarrow G_{J'}\subsetneq G_J\,$. By
part 2), $I_{G_J}$ is the smallest of all $\cO_L$-ideals $J'$ such that $G_J
\subseteq G_{J'}\,$. Therefore $J=I_{G_J}\,$, which shows that $J$ is a ramification
ideal.

\sn
7): Since $\cE$ is a nontrivial unibranched extension, also $G=G_{\cO_L}$ is nontrivial,
which by definition of $I_G$ implies that $I_G$ is nonzero. Since $G=G_{\cO_L}$, we
have $I_G\subseteq \cO_L\,$. Thus $I_G$ is a ramification ideal. As $\psi$ preserves
inclusion, $I_G$ is the largest ramification ideal of $\cE$.

If in addition $\cE$ is of prime degree, then the only subgroups of $G$ are $G$ and
$\{\rmid\}$. Since $I_{\{\trmid\}}=(0)$ is not a ramification ideal, $I_G$ is then the
unique ramification ideal of $\cE$.
\end{proof}

\pars
By means of the bijection defined in (\ref{vIS}) and (\ref{vSI}), the upper
ramification groups can be represented as
\[
G_\Sigma\>:=\>G_{I_\Sigma}\>=\>
\left\{\sigma\in G\>\left|\;\> v\,\frac{\sigma b-b}{b}\in \Sigma\cup\{\infty\}
\mbox{ \ for all }b\in L^\times\right.\right\} \>,
\]
where $\Sigma$ runs through all final segments of $(vL)^{\geq 0}$ and $\emptyset$.

Like the function (\ref{eq12}), also the function $\Sigma\mapsto G_\Sigma$ is in 
general neither injective nor surjective. We call a final segment $\Sigma$
of $(vL)^{\geq 0}$ a \bfind{ramification jump} if and only if
\[
\Sigma'\subsetneq\Sigma\>\Rightarrow\> G_{\Sigma'}\subsetneq G_\Sigma\>
\]
for every final segment $\Sigma'$ of $(vL)^{\geq 0}$.

In \cite{KuRz} and \cite{pr1}, ramification ideals are defined by use of ramification
jumps. The following result shows the equivalence with our above definition.
\begin{proposition}                         \label{rj}
Take a nontrivial unibranched Galois extension $\cE=(L|K,v)$. Then a nonempty final
segment $\Sigma$ of $(vL)^{\geq 0}$ is a ramification jump if and only if $I_\Sigma$
is a ramification ideal.
\end{proposition}
\begin{proof}
First note that for every nonempty final segment $\Sigma$ of $(vL)^{\geq 0}$ the ideal
$I_\Sigma$ is nonzero and contained in $\cM_L\,$. Now
a nonempty final segment $\Sigma$ of $(vL)^{\geq 0}$ is a ramification jump if and
only if for every nonempty final segment $\Sigma'$ of $(vL)^{\geq 0}$ we have $\Sigma'
\subsetneq \Sigma \Rightarrow G_{I_{\Sigma'}}\subsetneq G_{I_\Sigma}$. This holds
if and only if for every nonzero $\cO_L$-ideal $J'$ we have $J'\subsetneq
I_\Sigma\Rightarrow G_{J'}\subsetneq G_{I_\Sigma}$. By part 6) of Proposition~\ref{ris},
this in turn holds if and only if $I_\Sigma$ is a ramification ideal.
\end{proof}

\pars
By Propositions~\ref{ris} and~\ref{rj}, the number of ramification ideals and
ramification jumps in a unibranched Galois extension is bounded from above by the number
of nontrivial subgroups of its Galois group. It may not always be equal to this number,
as an example given in Section~\ref{sect2ex} below will show. For computations of the
number of ramification ideals in certain Galois extensions, see \cite{DN}.

\begin{remark}                  \label{firstdefI_E}
In this paper we are particularly interested in the case where $\cE=(L|K,v)$ is a purely
wild, or at least a unibranched, Galois extension of prime degree $p$. Then by part 7)
of Proposition~\ref{ris},
$\cE$ has the unique ramification ideal $I_G\,$, and we denote it by $I_\cE\,$. Hence
$\Sigma_\cE:=\Sigma_{I_\cE}$ is the unique ramification jump of $\cE$. As we will
show in what follows, ramification jump and ramification
ideal carry important information about the extension $(L|K,v)$.
\end{remark}

\begin{remark}
\rm In \cite[Section 2.1]{KuRz} we also included zero ideals and empty segments in the
definitions of the functions (\ref{vIS}) and (\ref{vSI}). However, classically
ramification jumps have always been defined as integers in the case of discrete 
valuations, and as real numbers in the case of valuations of rank one, i.e., their
value group is archimedean ordered. This means that it can be embedded in the ordered
additive group of the reals.
However, the intended meaning of ``jump'' does not fit well with the value $v0=\infty$.
\end{remark}

\pars
Further, we want to quickly discuss the \bfind{series of lower ramification groups}
\begin{equation}
G^l_I\>:=\>\left\{\sigma\in G\>\left|\;\> \sigma b -b\in I
\mbox{ for all }b\in {\cal O}_L\right.\right\} \>,
\end{equation}
where $I$ runs through all ${\cal O}_L$-ideals (see \cite[\$12]{ZS2}). Note that
$G^l_{\cM_L}$ is the inertia group of $(L|K,v)$. Again, for every ${\cal O}_L$-ideal
$I$, $G^l_I$ is a normal subgroup of $G$ (see \cite[(d) on p.79]{ZS2}), and $G_I
\subseteq G^l_I$ (see \cite[(a) on p.78]{ZS2}).
But in the case of an immediate extension $(L|K,v)$, the two groups coincide, as
follows from the next, more general, result:
\begin{lemma}
If $vL=vK$, then $G_I = G^l_I$ for all nonzero ideals $I$ of
${\cal O}_L$ contained in $\cM_L\,$.
\end{lemma}
\begin{proof}
It suffices to show that $G^l_I\subseteq G_I\,$. Take $\sigma\in G^l_I$
and $b\in L^\times$. Since $vL=vK$, we can pick some
$c\in K$ such that $vcb=0$. As $\sigma\in G^l_I\,$, we have that
$\sigma(cb)-cb \in I$. Since $vcb=0$, it follows that
\[
\frac{\sigma b -b}{b} \>=\> \frac{\sigma (cb) -cb}{cb} \>\in\>I\>.
\]
This shows that $\sigma\in G_I\,$.
\end{proof}

\mn
%
%
\subsection{Valuation bases}           \label{sectvb}
\mbox{ }\sn
Take an extension $(L|K,v)$. Elements $b_1,\ldots,b_n\in L$ are called
\bfind{valuation independent} (over $K$) if for all choices of $c_1,\ldots,c_n\in K$,
\[
v\sum_{i=1}^nc_i b_i\>=\>\min_i vc_i b_i\;.
\]
If in addition these elements form a basis of $L|K$, then they are called a
\bfind{valuation basis} of $(L|K,v)$.
If the valuation basis contains 1, we will speak of a \bfind{valuation basis with 1}. 

Recall that a unibranched extension $(L|K,v)$ is defectless if it satisfies the
fundamental equality $[L:K]={\rm e} \cdot {\rm f}$, where e$\,=(vL:vK)$
is the ramification index and f$\,= [Lv:Kv]$ is the inertia degree. In
this case, $(L|K,v)$ admits a \bfind{standard valuation basis}, which we
construct as follows: we take $y_1,\ldots,y_{\rm e}\in L$ such that
$vy_1+vK, \ldots, vy_{\rm e}+vK$ are the cosets of $vK$ in $vL$, and
$z_1,\ldots,z_{\rm f}\in L$ such that $z_1v,\ldots,z_{\rm f}v$ are a
basis of $Lv|Kv$. Then the products $y_iz_j\,$, $1\leq i\leq\,$e, $1\leq j\leq\,$f,
form a valuation basis of $(L|K,v)$ (see \cite[Lemma~3.2.2]{EP}). Note that we can 
always choose $y_1=z_1=1$ so that $y_1z_1=1$. We will then speak of a 
\bfind{standard valuation basis with 1}. 

The next result has been shown in the proof of \cite[Lemma~2.1]{Ku30}.
\begin{lemma}                            \label{vbdegp}
Take a unibranched extension $(L|K,v)$ of prime degree $p$. If for $b\in L^\times$,
either $vb\notin vK$ or there is some $c\in K$ such that $vcb=0$ and $cbv\notin Kv$,
then $1,b,\ldots,b^{p-1}$ forms a standard valuation basis with 1 of $(L|K,v)$.
\end{lemma}

For the following, cf.\ \cite[Proposition~3.4]{Ku58}.
\begin{lemma}                               \label{dliffexvb}
Take a finite unibranched extension $(L|K,v)$. Then the following are equivalent:
\sn
a) $(L|K,v)$ is defectless,
\n
b) $(L|K,v)$ admits a valuation basis,
\n
c) $(L|K,v)$ admits a standard valuation basis,
\n
d) $(L|K,v)$ admits a standard valuation basis with 1.
\end{lemma}
\begin{proof}
Implication a)$\Rightarrow$d) has just been shown above. Implications d)$\Rightarrow$c)
and c)$\Rightarrow$b) are trivial. For the implication b)$\Rightarrow$a), see the proof 
of \cite[Proposition~3.4]{Ku58}.
\end{proof}
In particular, for a finite unibranched defectless extension there is always a valuation 
basis with 1.

\begin{lemma}                               \label{apprdle}
Take a finite unibranched defectless extension $(L|K,v)$ and $a\in L$.
Then the set $\{v(a-c)\mid c\in K\}$ has a maximum. More precisely, if
we choose a valuation basis $b_1,\ldots,b_n$ for $(L|K,v)$ with $b_1=1$ and write
\[
a\>=\>\sum_{i=1}^n c_i b_i\>,
\]
then $v(a-c_1)$ is the maximum of $\{v(a-c)\mid c\in K\}$.
\end{lemma}
\begin{proof}
For every $c\in K$,
\begin{eqnarray*}
v(a-c_1) & = & v\sum_{i=2}^n c_i b_i
\>=\>\min_{2\leq i\leq n} vc_i b_i
\>\geq\> \min\{v(c_1-c)\,,\,vc_i b_i\mid 2\leq i\leq n\}\\
 & = & v\left(c_1-c+\sum_{i=2}^n c_i b_i\right) \>=\>v(a-c)\>.
\end{eqnarray*}
\end{proof}

\begin{corollary}                               \label{corapprdle}
Take a unibranched defectless extension of prime degree $p$ and $a_0\in L\setminus K$.
Then there is some $c\in K$ such that for $a=a_0-c$, the elements $1,a,\ldots,a^{p-1}$
form a valuation basis.
\end{corollary}
\begin{proof}
By Lemma~\ref{apprdle} there is some $c\in K$ such that $v(a_0-c)=\max\{v(a_0-c)\mid 
c\in K\}$. By Lemma~\ref{immb-c} this can only happen if either $v(a_0-c)\notin vK$ or 
there is some $d\in K$ such that $vd(a_0-c)=0$ and $d(a_0-c)v\notin Kv$. We 
set $a=a_0-c$; then in both cases, the elements $1,a,\ldots,a^{p-1}$ form a valuation 
basis by Lemma~\ref{vbdegp}.
\end{proof}

For a more general setting, see Lemma~2.10 and Corollary 2.11 of \cite{Ku58}.

\mn
%
%
\section{Computation of ramification ideals}            \label{sectcompri}
%
%

%
%
%
\subsection{Basic computations}                         \label{sectann}
\mbox{ }\sn
We are going to prove:
\begin{proposition}                            \label{dlrip}
Take a finite unibranched defectless Galois extension $\cE=(L|K,v)$ with Galois group $G$.
Then every ramification ideal is principal.
\end{proposition}

Take a nontrivial subgroup $H$ of $G$. We will prove the proposition by giving an 
algorithm for the computation of an element $b\smin$ such that for some $\sigma\in H$,
\begin{equation}                    \label{bmin}
v\left(\frac{\sigma b\smin}{b\smin}-1\right)\>=\>\min\left\{\left. v\left(
\frac{\sigma b}b -1\right)\,\right|\, \sigma\in H\,,\,b\in L^\times \right\}\>,
\end{equation}
which means that $\frac{\sigma b\smin}{b\smin}-1$ generates the ramification ideal
(\ref{genI}).

\begin{remark}
\rm This proposition was proven in 1970 by P.~Ribenboim in \cite{R2} for rank one
valuations. He uses that after embedding $vL$ in $\R$, every subset of
$vL$ bounded from below has an infimum in $\R$. Our computations presented
below are inspired by his. As they will show, the assumption of rank one is not
necessary. The use of infima can be replaced by the use of final segments of $vL$.
In the same way, results proven by Ribenboim in \cite{R} for rank one have been
generalized to arbitrary rank in \cite{Kucuts}.

A different version of the computation was presented by M.~Marshall in \cite{M}. He does
not assume that $(L,v)$ has rank one, but that $(K,v)$ is maximally complete and that the
extension $Lv|Kv$ is separable.
The assumption that $(K,v)$ is maximally complete means that it has no nontrivial
immediate extensions, and this implies that $(K,v)$ is defectless and henselian.
\pars
In \cite{R1} Ribenboim attempts to prove Proposition~\ref{dlrip} for all unibranched
Galois extensions with non-discrete rank one valuations without the ``defectless''
assumption, but this is false. Indeed, take Example 3.12 of a defect extension of prime
degree from \cite{Ku31} (originally, Abhyankar gave this example for a different purpose
without noticing the defect). By part 3) of Corollary~\ref{ASdef} below, its unique
ramification ideal is a nonprincipal $\cO_L$-ideal.

Ribenboim's mistake was noticed by J.~L.~Chabert. In
\cite{R2} Ribenboim then gives a correct proof of Proposition~\ref{dlrip}
for all finite defectless unibranched Galois extensions in the case of rank one
valuations. Curiously, his false proposition from \cite{R1} is later repeated in
\cite{R3} with literally the same incorrect proof.
\end{remark}

We shall now present computations that will not only prove the above proposition, but
will also be used later for more advanced results. Let us start with some useful
basic principles.
\begin{lemma}                                    \label{lemv(sai/ai-1)}
Let $K$ be any field and take $\sigma\in \Gal K$.
\sn
1) For all $a,b\in \tilde{K}^\times$ and $c\in K$,
\begin{equation}                         \label{sab/ab-1}
\frac{\sigma cab}{cab}-1 \>=\> \frac{\sigma ab}{ab}-1 \>=\>
\left(\frac{\sigma a}{a}-1\right)
\left(\frac{\sigma b}{b}-1\right)\>+\>\left(\frac{\sigma a}{a}-1\right)
\>+\>\left(\frac{\sigma b}{b}-1\right)\>.
\end{equation}
\sn
2) Assume that $v$ is a valuation on $\tilde K$ and that $a\in \tilde K$ is such that $v\left(\frac{\sigma a}{a}-1\right)>0$. Take $i\in\N$ and assume that $i<\chara Kv$ if
$\chara Kv>0$. Then
\begin{equation}                                 \label{v(sai/ai-1)}
v\left(\frac{\sigma a^i}{a^i}-1\right) \>=\> v\left(\frac{\sigma a}{a}-1\right) \>.
\end{equation}
\end{lemma}
\begin{proof}
1): We leave the straightforward proof to the reader.
\sn
2): By our assumption on $i$, we have $vi=0$. Using this together with equation
(\ref{sab/ab-1}), one proves equation (\ref{v(sai/ai-1)}) by induction on $i$.
\end{proof}

Further, we will need the following facts. We note that if $(L|K,v)$ is a unibranched
Galois extension, then for every $\sigma\in\Gal L|K$ and $b\in L^\times$,
\begin{equation}             \label{inOL}   
\frac{\sigma b}{b} - 1\>\in\> \cO_L\>.
\end{equation}

\begin{lemma}
Assume that $(L|K,v)$ is a purely wild Galois extension. Then
for every $\sigma\in \Gal L|K$ and all $a,b\in L^\times$,
\begin{equation}             \label{inML}
\frac{\sigma b}{b} - 1\>\in\> \cM_L\>
\end{equation}
and
\begin{equation}                      \label{inML2}
v\left(\frac{\sigma cab}{cab}-1\right) \>\geq\> \min\left\{v\left(\frac{\sigma a}{a}-1
\right),\,v\left(\frac{\sigma b}{b}-1\right)\>\right\}\>,
\end{equation}
with equality holding if $v\left(\frac{\sigma a}{a}-1\right)\ne v\left(\frac{\sigma b}
{b}-1\right)$.
\end{lemma}
\begin{proof}
Equation (\ref{inML}) holds since by the definition of ``purely wild extension'',
$\Gal L|K=G_{\cM_L}\,$. Equation (\ref{inML2}) follows from equations (\ref{sab/ab-1})
and (\ref{inML}).
\end{proof}

\begin{lemma}                    \label{baibilem}
Take a normal unibranched extension $(L|K,v)$ and an intermediate field $L_0\,$.
Assume that the elements $b_1\,,\ldots,b_n\in L$ are valuation independent over
$L_0$ and write
\begin{equation}                 \label{bsum}
0\>\ne\>b\>=\>\sum_{i=1}^n a_ib_i\>  
\end{equation}
with elements $a_1,\ldots,a_n\in L_0\,$.
Then for each automorphism $\sigma$ of $L$ over $K$, not necessarily trivial on $L_0\,$,
\begin{equation}                    \label{baibi}
v\left(\frac{\sigma b}{b} - 1\right)\>\geq\>
\min_i v\left(\frac{\sigma a_ib_i}{a_ib_i}-1\right)\>
\end{equation}
where the minimum runs over all $i$ for which $a_i\ne 0$.
If in addition $\sigma$ is trivial on all $b_i\,$, then $\frac{\sigma b}{b}-1$ lies
in the $\cO_L$-ideal generated by the elements $\frac{\sigma a_i}{a_i}-1$. On the other
hand, if $\sigma$ is trivial on all $a_i\,$, then $\frac{\sigma b}{b}-1$ lies
in the $\cO_L$-ideal generated by the elements $\frac{\sigma b_i}{b_i}-1$.
\end{lemma}
\begin{proof}
We have
\begin{equation}
\frac{\sigma b}{b} - 1\>=\> \sum_i \left(\frac{\sigma a_ib_i}{a_i b_i}-1\right) \cdot
\frac{a_ib_i}{b}\>
\end{equation}
where the sum runs over all $i$ for which $a_i\ne 0$. Since $vb\leq va_ib_i=v\sigma
a_ib_i$ for $1\leq i\leq n$, this implies (\ref{baibi}) and that
$\frac{\sigma b}{b}-1$ lies in the $\cO_L$-ideal generated by the elements
$\frac{\sigma a_ib_i}{a_ib_i}-1$. If in addition $\sigma b_i=b_i$ for all $i$,
then $\frac{\sigma a_ib_i}{a_ib_i}-1=\frac{\sigma a_i}{a_i}-1$. On the other hand, if
$\sigma a_i=a_i$ for all $i$, then $\frac{\sigma a_ib_i}{a_ib_i}-1=
\frac{\sigma b_i}{b_i}-1$.
\end{proof}

\begin{corollary}                       \label{part1&2}
Assume that $\cE=(L|K,v)$ is a finite defectless unibranched Galois extension with
Galois group $G$.
\sn
1) Choose a valuation basis $b_i\,$, $1\leq i\leq\,n$. Then for all $b\in L^\times$,
\begin{equation}                         \label{gamma=min}
v\left(\frac{\sigma b}{b}-1\right)\>\geq\>\gamma_\cE\>:=\>
\min\left\{\left. v\left(\frac{\sigma b_i}
{b_i}-1\right)\,\right|\,\sigma\in G\,,\, 1\leq i\leq n\right\}\>\geq\> 0\>.
\end{equation}
Hence $I_G$ is principal with generator $b\smin$ which can be chosen to be $b_i$ for
suitable $i$. Then $vb\smin=\gamma_\cE\,$.
\sn
2) Assume in addition that $\cE$ is purely wild and
choose a standard valuation basis $y_iz_j\,$, $1\leq i\leq\,$e, $1\leq j\leq\,$f of
$(L|K,v)$ as described in Section~\ref{sectvb}. Then for all $b\in L^\times$,
\[
v\left(\frac{\sigma b}{b}-1\right)\>\geq\> \min\left\{\left. v\left(\frac{\sigma y_i}
{y_i}-1\right),\,v\left(\frac{\sigma z_j}{z_j}-1\right) \,\right|\, \sigma\in G\,,\,1\leq
i\leq\rme,\, 1\leq j\leq\rmf \right\}\>.
\]
\end{corollary}
\begin{proof}
1): It follows from (\ref{inOL}) that no values in (\ref{gamma=min}) are less than $0$.
Since $\sigma\in G$ is trivial on all $a\in L_0\,$, our statement follows from
Lemma~\ref{baibilem}, where we take $L_0=K$.
\mn
2): Take $b\in L$ and write it in the form $b=\sum_{i,j} c_{i\,j}y_iz_j$. Then our
assertion follows from part 1) of our corollary together with (\ref{inML2}).
\end{proof}

\pars
\begin{proposition}                      \label{compri}
Assume that $(L|K,v)$ is a finite purely wild Galois extension with Galois group $G$.
Assume further that $L_0$ is an intermediate field of $L|K$ such that $(L|L_0,v)$ is
defectless and that $b_i\,$, $1\leq i\leq\,n$ is a valuation basis of $(L|L_0,v)$. Define
\begin{equation}                         
\gamma\>:=\>\min\left\{\left. v\left(\frac{\sigma b_i}
{b_i}-1\right)\,\right|\,\sigma\in G\,,\, 1\leq i\leq n\right\}\> >\> 0\>.
\end{equation}
Assume further that there is $\gamma_0\in vL$ such that
\begin{equation}                          \label{gamma0}
v\left(\frac{\sigma a}{a} - 1\right)\>\geq\>\gamma_0  \quad \mbox{for all $\sigma\in G$
and $a\in L_0^\times$}\>.
\end{equation}
Then
\begin{equation}                          \label{gamma=min0}
v\left(\frac{\sigma b}{b} - 1\right)\>\geq\>\min\{\gamma_0\,,\,\gamma\}  \quad
\mbox{for all $\sigma\in G$ and $b\in L^\times$}\>.
\end{equation}
If $\gamma\leq\gamma_0\,$, then
\begin{equation}                          \label{gamma=minL}
\min\left\{\left. v\left(\frac{\sigma b}b -1\right)\,\right|\, \sigma\in G\,,\,b\in
L^\times \right\}\>=\>\gamma\>,
\end{equation}
hence $I_G$ is principal and $b\smin$ can be chosen to be $b_i$ for suitable $i$.
\end{proposition}
\begin{proof}
Take $b\in L^\times$ and write it in the form (\ref{bsum}) with $a_i\in L_0\,$. Using
(\ref{baibi}) together with (\ref{inML2}), we obtain:
\begin{eqnarray*}
v\left(\frac{\sigma b}{b} - 1\right)&\geq&
\min\left\{\left. v\left(\frac{\sigma a_i b_i}{a_i b_i}-1\right)\,\right|\,
\sigma\in G\,,\, a_i\in L_0^\times\,,\, 1\leq i\leq n\right\}\\
&\geq& \min\left\{\left. v\left(\frac{\sigma a_i}{a_i}-1\right),\,
v\left(\frac{\sigma b_i}{b_i}-1\right)\,\right|\, \sigma\in G\,,\, a_i\in L_0^\times\,,\,
1\leq i\leq n\right\}\\
&=& \min\{\gamma_0\,,\,\gamma\}\>,
\end{eqnarray*}
which proves (\ref{gamma=min0}). The last assertion is obvious.
\end{proof}

\sn
{\it Proof of Proposition~\ref{dlrip}:}\n
Take any nontrivial subgroup $H$ of $G_{\cM_L}$ and denote its fixed field in $L$
by $K_H$. Then also $\cE_H:=(L|K_H,v)$ is a finite unibranched Galois extension, by
Lemma~\ref{md} it is again defectless, and its Galois group is $H$. Applying part 1)
of Corollary~\ref{part1&2} with $\cE_H$ in place of $\cE$, equation (\ref{gamma=min})
yields that $I_H$ is principal. This proves our proposition. \qed

\pars
Finally, we
prove a generalization of a fact that has been used in \cite[Section~7.1]{Th2}.
A unibranched algebraic extension $(L|K,v)$ is called \bfind{tame} if every finite
subextension $K'|K$ satisfies the following conditions:
\sn
(T1) \ the ramification index $(vK':vK)$ is not divisible by $\chara Kv$,
\sn
(T2) \ the residue field extension $K'v|Kv$ is separable,
\sn
(T3) \ the extension $(K'|K,v)$ is defectless.
\sn
Every subextension of a tame extension is tame.
For detailed information on tame extensions and their connection with purely wild
extensions, see \cite{Ku39,Ku2}. We will need:
\begin{lemma}                               \label{tpw}
Take a henselian field $(K,v)$, a purely wild Galois extension $(L|K,v)$ and a tame Galois
extension $(K'|K,v)$. Let $L'$ be the field compositum $L.K'$ of $L$ and $K'$, endowed with
the unique extension of $v$ from $K$. Then $L|K$ is linearly disjoint from $K'|K$,
$(L'|K',v)$ is a purely wild Galois extension of degree $[L:K]$, and $(L'|L,v)$ is a tame
extension of degree $[K':K]$.
\end{lemma}
\begin{proof}
By \cite[Proposition 4.1]{Ku2}, an extension $(K'|K,v)$ is tame if and only if it is a
subextension of $(K^r|K,v)$, where $(K^r,v)$ is the absolute ramification field of
$(K,v)$. By  It follows that $L|K$ and $K'|K$ are linearly
disjoint, $[L':K']=[L:K]$ and $[L':L]=[K':K]$.
\pars
From \cite[(20.15), p.\ 166]{En} we know that $(L^r,v)=(L.K^r,v)$ and $((K')^r,v)
=(K'.K^r,v)$. Since $(K'|K,v)$ is tame, we have $K'\subseteq K^r$ and so,
$L'=L.K'\subseteq L.K^r=L^r$, showing that $(L'|L,v)$ is tame. Since $(L|K,v)$
is purely wild, $L|K$ is linearly disjoint from $K^r|K$. It follows that
$L'=L.K'$ is linearly disjoint from
$K'.K^r=(K')^r$ over $K'$, showing that $(L'|K',v)$ is purely wild.
\end{proof}

\begin{proposition}
Take a henselian field $(K,v)$, a finite purely wild Galois extension $(L|K,v)$ and a
tame Galois extension $(K'|K,v)$. Then with the unique extension of $v$ to the
compositum $L'=L.K'$, also $(L'|K',v)$ is a purely wild Galois extension of degree
$[L:K]$, and
\begin{equation}                     \label{ItoI'}
I\>\mapsto\>I\cO_{L'}
\end{equation}
is a bijection between the ramification ideals of $(L|K,v)$ and those of $(L'|K',v)$.
Its inverse is the function
\begin{equation}                     \label{I'toI}
I'\>\mapsto\>I'\cap\cO_{L}
\end{equation}
from the ramification ideals of $(L'|K',v)$ to those of $(L|K,v)$.
\end{proposition}
\begin{proof}
By Lemma~\ref{tpw}, $L|K$ and $K'|K$ are linearly disjoint. Therefore,
$L'|K'$ is a Galois extension with its Galois group $G$ isomorphic to $\Gal L|K$ via
the restriction
\[
\Gal L'|K'\ni\sigma\>\mapsto\>\sigma|_L\in\Gal L|K\>.
\]
For a subgroup $H'\subseteq\Gal L'|K'$, denote its image under the restriction by $H$. If
we start with a subgroup $H\subseteq\Gal L|K$, then the mentioned isomorphism yields
$H'\in\Gal L'|K'$.
\pars
Take a finite Galois subextension $(K'_0|K,v)$ of $(K'|K,v)$. It is again tame, hence by
Lemma~\ref{tpw}, $(L'_0|L,v)$ is a finite tame Galois extension, where $L'_0=L.K'_0\,$.
In particular, every valuation basis $b_1,\ldots,b_n$ of $(K'_0|K,v)$ is also a valuation
basis of $(L'_0|L,v)$.

Each element $b\in L'$ already lies in such a compositum $L'_0=L.K'_0$, so it can be
written as $b=\sum_{1\leq i\leq n} a_ib_i$ with $b_1\,,\ldots,b_n$ a valuation basis of
$(K'_0|K,v)$ and suitable elements $a_i\in L$.

Now take a nontrivial subgroup $H$ of $\Gal L|K$ and consider the ramification ideal
$I_{H'}$ of $(L'|K',v)$. Take $\sigma\in H'$ and $0\ne b\in L'_0$ such that
$\frac{\sigma b}{b}-1
\in I_{H'}$ and write $b$ in the form as indicated above. Since $\sigma$ is
trivial on $K'$, Lemma~\ref{baibilem} with $L'_0$ in place of $L$ and $L$ in
place of $K$ shows that $\frac{\sigma b}{b}-1$ lies in the
$\cO_{L'_0}$-ideal generated by the elements
\[
\frac{\sigma a_i}{a_i}-1\>=\>\frac{\sigma|_L\, a_i}{a_i}-1\,\in\, I_H\>.
\]
Therefore,
\begin{equation}
\frac{\sigma b}{b}-1\,\in\, I_H \cO_{L'_0}\>\subseteq\> I_H\cO_{L'}\>,
\end{equation}
which shows that the ramification ideal $I_{H'}$ of $(L'|K',v)$ is a subset of $I_H
\cO_{L'}\,$. To prove the reverse inclusion, take an element $\frac{\tau a}{a}-1
\in I_H\,$, where $\tau\in H$ and $a\in L^\times$. We write $\tau=\sigma|_L$ with
$\sigma\in H'$. Then
\[
\frac{\tau a}{a}-1\>=\>\frac{\sigma a}{a}-1\,\in\, I_{H'}\>.
\]
Thus $I_H\subseteq I_{H'}\,$, and we obtain
\[
I_H\cO_{L'}\>=\> I_{H'}\>.
\]
This proves that the function (\ref{ItoI'}) sends ramification ideals of $(L|K,v)$ to
ramification ideals of $(L'|K',v)$. It also shows that $I_{H'}$ is the collection of all
elements in $L'$ whose value is not less than the value of some element in $I_H\,$.
This implies that $I_{H'}\cap\cO_L$ is the collection of all elements in $L$ whose value
is not less than the value of some element in $I_H\,$. In other words, $I_{H'}\cap\cO_L=
I_H\cO_L=I_H\,$. Hence, $I_H\cO_{L'}\cap\cO_L=I_H$, which proves that the function
(\ref{ItoI'}) is a bijection, with the function (\ref{I'toI}) its inverse.
\end{proof}

\begin{remark}
\rm In \cite[Section~7.1]{Th2} only the special case is considered where $(K,v)$ is
a henselian field of mixed characteristic, $L|K$ has prime degree $p$ and
$K'=K(\zeta_p)$ where $\zeta_p$ is
a $p$-th root of unity. The latter implies that $(K'|K,v)$ is a tame extension. This
case is of interest when $L|K$, though being Galois, is not a Kummer extension, since
$L'|K'$ will be a Kummer extension.
\end{remark}

With a proof adapted from the one of the previous proposition, the following can be
shown:
\begin{proposition}
Take a henselian field $(K,v)$, a finite immediate Galois extension $(L|K,v)$ and a
Galois extension $(K'|K,v)$ for which every finite subextension is defectless. Then
with the
unique extension of $v$ to the compositum $L'=L.K'$, also $(L'|K',v)$ is an immediate
Galois extension of degree $[L:K]$, and (\ref{ItoI'}) is again a bijection between
the ramification ideals of $(L|K,v)$ and those of $(L'|K',v)$.  \qed
\end{proposition}

\mn
%
%
%
\subsection{Ramification ideals and defect}                   \label{sectridef}
\mbox{ }\sn
We start with the following easy observation. If $(K(a)|K,v)$ is any algebraic
extension of valued fields and $\sigma\in \Gal(K)$ is such that $\sigma a\ne a$, then
\[
\left\{\left. v\>\frac{\sigma(a-c)-(a-c)}{a-c}
\>\right|\> c\in K\right\}\>=\>
\left\{\left. v\>\frac{\sigma a-a}{a-c}
\>\right|\> c\in K\right\}\>=\> -v(a-K)+v(\sigma a-a)\>.
\]

\begin{lemma}                               \label{anypol}
Take an extension $(K(a)|K,v)$ of prime degree and any $b\in K(a)^\times$. Assume that
$v(a-K)$ has no maximal element. Then for all $\sigma \in \Gal(K)$  such that $\sigma
a\ne a$ and $v(a-K)<v(\sigma a-a)$, there is some $c\in K$ such that
\begin{equation}                            \label{anypole}
v\,\frac{\sigma b -b}{b} \;>\; -v(a-c)\,+\, v(\sigma a-a)\;.
\end{equation}
\end{lemma}
\begin{proof}
Since we are already assuming that $v(a-K)$ has no maximal element and that $v(a-K)<
v(\sigma a-a)$, the proof can be taken over from the proof of \cite[Lemma 3.3]{KuRz},
using \cite[Lemma 2.11]{KuRz}.
\end{proof}

\pars
Now take a Galois extension $\cE=(L|K,v)$ of prime degree $p$ with Galois group
$G$. For $\sigma\in G\setminus\{\rmid\}$ we set
\begin{equation}                        \label{Sigsig}
\Sigma_\sigma\>:=\> \left\{ v\left( \left.\frac{\sigma b-b}{b}\right) \right| \, 
b\in L^{\times} \right\} \>.
\end{equation}

\begin{theorem}                                        \label{SigmaE'}
Assume that $a\in L$ is a generator of $\cE$ such that $v(a-K)$ has no maximal element.
Then for every $\sigma \in G\smid$ such that $v(a-K)<v(\sigma a -a)$,
\begin{equation}                                       \label{ram_gp'}
\Sigma_\sigma \>=\> -v(a-K)+v(\sigma a -a)\>,
\end{equation}
which is a final segment of $vL$.
\end{theorem}
\begin{proof}
Since $v(a-K)$ is an initial segment of $vK$ by Theorem~\ref{CutImm}, $-v(a-K)$ is a
final segment of $vK$, hence $-v(a-K)+v(\sigma a -a)$ is a final segment of $vL$.
Therefore, Equation (\ref{ram_gp'}) follows from Lemma~\ref{anypol} together with our
observation at the start of the section.
\end{proof}

We set
\begin{equation}                        \label{SigE}
\Sigma_\cE\>:=\> \bigcup_{\trmid\ne\sigma\in G} \Sigma_\sigma \>=\>
\left\{ v\left( \left.\frac{\sigma b-b}{b}\right) \right| \,
\mbox{id}\ne\sigma\in G\,,\,b\in L^{\times} \right\} \>.
\end{equation}
As a union of final segments of $vK$, $\Sigma_\cE$ is itself a final segment of $vK$.
Thus we can set
\begin{equation}                        \label{I_E}
I_{\cE}\>:=\> I_{\Sigma_\cE}\>.
\end{equation}
If all $\Sigma_\sigma$ are bouded from below, then so is $\Sigma_\cE$ and $I_\cE$
is a (possibly fractional) $\cO_L$-ideal.
\begin{proposition}                              \label{I_Eramid}
Assume that $a\in L$ is a generator of $\cE$ such that $v(a-K)$ has no maximal element.
\sn
1) If $\cE$ is unibranched, then $v(\sigma a-a)$ and $\Sigma_\sigma$ are independent of
the choice of $\sigma \in G\smid$.
\sn
2) With $\Sigma_\cE$ and $I_\cE$ as in (\ref{SigE}) and (\ref{I_E}), the following
assertions are equivalent:
\sn
a) for all $\sigma \in G\smid$, $v(a-K)$ is bounded from above by $v(\sigma a -a)$,
\sn
b) $\Sigma_\cE\subseteq vL^{>0}$,
\sn
c) $I_\cE$ is a higher ramification ideal,
\sn
d) $\cE$ is purely wild,
\sn
e) $\cE$ is an immediate defect extension,
\sn
f) $\cE$ is unibranched.
\end{proposition}
\begin{proof}
1): Take any $\sigma \in G\smid$. Since $\cE$ is unibranched, we have $v=v\sigma$ on
$L$. Thus, $v(\sigma^2 a-\sigma a)=v\sigma(\sigma a-a)=v(\sigma a-a)$ and therefore,
\[
v(\sigma^2 a-a)\>\geq\>\min\{v(\sigma^2 a-\sigma a),v(\sigma a-a)\}\>=\>v(\sigma a-a)\>.
\]
Assume that $k<p-1$ and we have already shown that $v(\sigma^k a-a)\geq v(\sigma a-a)$.
Then $v(\sigma^{k+1} a-\sigma a)=v\sigma(\sigma^k a-a)=v(\sigma^k a-a)\geq
v(\sigma a-a)$ and therefore,
\[
v(\sigma^{k+1} a-a)\>\geq\>\min\{v(\sigma^{k+1} a-\sigma a),v(\sigma a-a)\}\>\geq\>
v(\sigma a-a)\>.
\]
By induction, it follows that $v(\sigma^i a-a)\geq v(\sigma a-a)$ for all $i\leq p-1$.
Since both $\sigma$ and $\sigma^i$ for $i\leq p-1$ are generators of $G$ and we had
chosen an arbitrary generator $\sigma$, it follows that the value $v(\sigma a-a)$ and
hence also $\Sigma_\sigma$ is independent of our choice of $\sigma$.
\sn
2): The proof of the equivalence of assertions a), b), c) and d) is straightforward.
\sn
d)$\Rightarrow$e): The purely wild extension $\cE$ cannot be defectless, because then
Lemma~\ref{apprdle} shows that $v(a-K)$ would have a maximal element for every $a\in
L\setminus K$. Hence $\cE$ is a defect extension, and since it has prime degree, it is
immediate.
\sn
e)$\Rightarrow$f): By the definition of defect extensions.
\sn
f)$\Rightarrow$e): By Lemma~\ref{imm_deg_p}, a unibranched extension $(K(a)|K,v)$
such that $v(a-K)$ has no maximal element is an immediate defect extension.
\sn
e)$\Rightarrow$d): Every defect extension $(L|K,v)$ of prime degree $p$ is immediate,
hence $(vL:vK)$ is a $p$-power and $Lv|Kv$ is purely inseparable, which shows that it
is purely wild.
\end{proof}

In Remark~\ref{firstdefI_E}, for a unibranched Galois extension $\cE=(L|K,v)$ of prime
degree, $I_\cE$ was defined to be $I_G\,$, and we set $\Sigma_\cE:=\Sigma_{I_\cE}\,$. If
$\cE$ has a generator $a\in L$ such that $v(a-K)$ has no maximal element, then by
implication f)$\Rightarrow$c) of the proposition, the ideal $I_\cE$ as defined in
(\ref{I_E}) is a higher ramification ideal, and as $\cE$ has a unique ramification ideal,
it must coincide with the ideal defined in Remark~\ref{firstdefI_E}.

\begin{corollary}                                   \label{SigmaEcor}
Assume that $a\in L$ is a generator of $\cE$ such that $v(a-K)$ has no maximal element,
and that $\Sigma_\cE\subseteq vL^{>0}$. Then $\cE$ is unibranched and an immediate defect
extension, and its unique ramification ideal is the nonprincipal ideal
\begin{equation}                                  \label{IEdef}
I_{\cE}\>=\> \left(\frac{\sigma_0\, a -a}{a-c}\>\left|\;\>
c\in K\right. \right) \>=\> \left(\frac{\sigma_0(a-c)}{a-c}-1\>\left|\;\>
c\in K\right. \right)\>,
\end{equation}
where $\sigma_0$ is any generator of $G$.
\end{corollary}
\begin{proof}
Under the assumptions of the corollary, implications b)$\Rightarrow$f) and
b)$\Rightarrow$e) of the previous proposition show that $\cE$ is unibranched and an
immediate defect extension. Implication b)$\Rightarrow$a) shows that $\Sigma_\cE=
\Sigma_\sigma$ for each $\sigma \in G\smid$; this proves (\ref{IEdef}). Since
$\Sigma_\cE$ has no smallest element, $I_\cE$ is nonprincipal.
\end{proof}

\mn
%
%
%
\subsection{Unibranched Galois extensions of prime degree} \label{sectGaldefdegp}
\mbox{ }\sn
A Galois extension of degree $p$ of a field $K$ of characteristic $p>0$ is an
\bfind{Artin-Schreier extension}, that is, generated by an \bfind{Artin-Schreier
generator} $\vartheta$ which is the root of an \bfind{Artin-Schreier
polynomial} $X^p-X-c$ with $c\in K$. A Galois extension of degree $p$ of a field 
$K$ of characteristic 0 which contains all $p$-th roots of unity is a \bfind{Kummer
extension}, that is, generated by a \bfind{Kummer generator} $\eta$ which satisfies 
$\eta^p\in K$. For these facts, see \cite[Chapter VIII, \S8]{L}.

If $(L|K,v)$ is a Galois defect extension of degree $p$ of fields of characteristic 
0, then a Kummer generator of $L|K$ can be chosen to be a $1$-unit. Indeed, choose 
any Kummer generator $\eta$. Since $(L|K,v)$ is immediate, we have that $v\eta\in
vK(\eta)=vK$, so there is $c\in K$ such that $vc=-v\eta$. Then $v\eta c=0$, and since 
$\eta cv\in K(\eta)v=Kv$, there is
$d\in K$ such that $dv=(\eta cv)^{-1}$. Then $v(\eta cd)=0$ and $(\eta cd)v=1$. Hence 
$\eta cd$ is a $1$-unit. Furthermore, $K(\eta cd)=K(\eta)$ and $(\eta cd)^p=
\eta^pc^pd^p\in K$. Thus we can replace $\eta$ by $\eta cd$ and assume from the start
that $\eta$ is a $1$-unit. It follows that also $\eta^p\in K$ is a $1$-unit.

Throughout this article, whenever we speak of ``Artin-Schreier extension'' we refer to
fields of positive characteristic, and with ``Kummer extension'' we refer to fields of
characteristic 0.

\sn
%
%
%
\subsubsection{The defectless case}                         \label{sectdefl}
\mbox{ }\sn
%
The following proposition is taken from \cite{pr1}. For the convenience of the reader, 
and as an illustration of the usefulness of Lemma~\ref{apprdle}, we include its proof 
here.

\begin{proposition}                               \label{ASKgen}
1) Take a valued field $(K,v)$ of equal positive characteristic $p$ and a unibranched
defectless Artin-Schreier extension $(L|K,v)$. 

If $\rmf(L|K,v)=p$, then the extension has an Artin-Schreier generator $\vartheta$ of 
value $v\vartheta\leq 0$ such that $Lv=Kv(\tilde{c}\vartheta v)$ for every $\tilde{c}
\in K$ with $v\tilde{c}\vartheta=0$; the extension $Lv|Kv$ is separable if and only if
$v\vartheta=0$. 

If $\rme(L|K,v)=p$, then the extension has an Artin-Schreier generator $\vartheta$ such 
that $vL=vK+\Z v\vartheta$. Every such $\vartheta$ satisfies $v\vartheta<0$.

\sn
2) Take a valued field $(K,v)$ of mixed characteristic and a unibranched defectless 
Kummer extension $(L|K,v)$ of degree $p=\chara Kv$. Then the extension has a Kummer 
generator $\eta$ such that: 
\sn
a) if $\,\rmf(L|K,v)=p$, then either $\eta v$ generates the residue field extension, 
in which case it is inseparable, or $\eta$ is a $1$-unit and for some $\tilde{c}\in K$, 
$\tilde{c}(\eta-1)v$ generates the residue field extension;
\sn
b) if $\,\rme(L|K,v)=p$, then either $v\eta$ generates the value group extension, or 
$\eta$ is a $1$-unit and $v(\eta-1)$ generates the value group extension.
\end{proposition}
\begin{proof}
1): Take any Artin-Schreier generator $y$ of $(L|K,v)$. Then by 
Lemma~\ref{apprdle} there is $c\in K$ such that either $v(y-c)\notin vK$, or
for every $\tilde{c}\in K$ such that $v\tilde{c}(y-c)=0$ we have $\tilde{c}(y-c)v\notin
Kv$. Since $p$ is prime, 
in the first case it follows that $\rme(L|K,v)=p$ and that $v(y-c)$ generates the value
group extension. In the second case it follows that $\rmf(L|K,v)=p$ and that $\tilde{c}
(y-c)v$ generates the residue field extension. In both cases, $\vartheta=y-c$ is an 
Artin-Schreier generator. Let $\vartheta^p-\vartheta=b\in K$.

Assume that $\rmf(L|K,v)=p$. If $v\vartheta<0$, then $v(\vartheta^p-b)=v\vartheta>
pv\vartheta=v\vartheta^p$, whence $v((\tilde{c}\vartheta)^p-\tilde{c}^pb)=
v\tilde{c}^p\vartheta>v(\tilde{c}\vartheta)^p$ for $\tilde{c}\in K$ with $v\tilde{c}\vartheta=0$ and therefore, $(\tilde{c}\vartheta)^p v
=\tilde{c}^pb v\in Kv$. In this case, the residue field 
extension is inseparable. Now assume that $v\vartheta\geq 0$ and hence also $vb\geq 0$.
The reduction of $X^p-X-b$ to $Kv[X]$ is a separable polynomial, so $Lv|Kv$ is 
separable. The polynomial $X^p-X-bv$ cannot have a zero in $Kv$, since otherwise 
the $p$ distinct roots of this polynomial give rise to $p$ distinct extensions of $v$
from $K$ to $L$, contradicting our assumption that $(L|K,v)$ is unibranched. 
Consequently, $bv\ne 0$, whence $vb=0$ and $v\vartheta= 0$.

Assume that $\rme(L|K,v)=p$. If $v\vartheta\geq 0$, then $vb\geq 0$ and $\vartheta v$
is a root of $X^p-X-bv$. If this polynomial does not have a zero in $Kv$, then
$\vartheta v$ generates a nontrivial residue field extension, contradicting our
assumption that $\rme(L|K,v)=p$. If the polynomial has
a zero in $Kv$, then similarly as before one deduces that
$(L|K,v)$ is not unibranched, contradiction. Hence $v\vartheta<0$.

\sn
2): Take any Kummer generator $y$ of $(L|K,v)$. If there is a Kummer generator $\eta$
such that $v\eta\notin vK$, then it follows as before that $\rme(L|K,v)=p$ and that 
$v\eta$ generates the value group extension. Now assume that there is no such $\eta$.

If there is a Kummer generator $y$ and some $\tilde{c}\in K$ such that $v\tilde{c}y=0$ 
and $\tilde{c}yv\notin Kv$, then it follows as before that $\rmf(L|K,v)=p$ and that 
$\tilde{c}yv$ generates the residue field extension. We set $\eta=\tilde{c}y$ and 
observe that also $\eta$ is a Kummer generator.
Since $(\eta v)^p\in Kv$, $Lv|Kv$ is purely inseparable in this case.

Now assume that the above cases do not appear, and choose an arbitrary Kummer generator 
$y$ of $(L|K,v)$. Consequently, we have that $vy\in vK$ and $\tilde{c}yv\in Kv$ for all 
$\tilde{c}\in K$ with $v\tilde{c}y=0$. Then as described at the start of this section, there 
are $c_1,c_2\in K$ such that $c_2c_1y$ is a Kummer generator of $(L|K,v)$ which is a 
$1$-unit. We replace $y$ by $c_2c_1y$.

By Lemma~\ref{apprdle} there is $c\in K$ such that $v(y-c)$ is maximal in 
$v(y-K)$ and either $v(y-c)\notin vK$ or there is some $\tilde{c}\in K$ 
such that $v\tilde{c}(y-c)=0$ and $\tilde{c}(y-c)v\notin Kv$. Since $y$ is a $1$-unit,
we know that $v(y-1)>0$, hence also $v(y-c)>0=vy$, showing that also $c$ is a $1$-unit.
Then $\eta:=c^{-1}y$ is again a Kummer generator of $(L|K,v)$ which is a $1$-unit. Since
$vc=0$, we know that $v(\eta-1)=vc(\eta-1)=v(y-c)$. Hence if $v(y-c)\notin vK$, then 
$v(\eta-1)$ generates the value group extension.

Now assume that there is $\tilde{c}\in K$ such that $v\tilde{c}(y-c)=0$ and 
$\tilde{c}(y-c)v\notin Kv$. Since $c$ is a $1$-unit, it follows that $v\tilde{c}(\eta-1)
=v\tilde{c}c(\eta-1)=v\tilde{c}(y-c)=0$ and $\tilde{c}(\eta-1)v=\tilde{c}c
(\eta-1)v=\tilde{c}(y-c)v$. We find that $\tilde{c}(\eta-1)v$ generates the residue field
extension.
\end{proof}

From this proposition we deduce:
\begin{theorem}                              \label{ThmASKgen}
Take a unibranched defectless Galois extension $\cE=(L|K,v)$ of prime degree $p$.
Let $\gamma_\cE$ be as defined in (\ref{gamma=min}).
\sn
1) If $\cE$ is an Artin-Schreier extension, then it admits an Artin-Schreier
generator $\vartheta$ of value $v\vartheta\leq 0$ such that $1,\vartheta,\ldots,
\vartheta^{p-1}$ form a valuation basis for $(L|K,v)$. The element $b\smin$ as in (\ref{bmin})
can be chosen to be $\vartheta$, so that 
\begin{equation}
I_{\cE}\>=\> \left(\frac{1}{\vartheta} \right)\>
\end{equation}
and $\gamma_\cE=-v\vartheta$.
We have $I_{\cE}=\cO_L$ if and only if $v\vartheta=0$, and this holds if and only if 
$Lv|Kv$ is separable of degree $p$.
\sn
2) Let $\cE$ be a Kummer extension, and fix a $p$-th root of unity $\zeta_p\,$.
Then there are two cases:
\sn
a) $\cE$ admits a Kummer generator $\eta$ such that $v\eta\geq 0$ and $1,\eta,
\ldots,\eta^{p-1}$ form a valuation basis for $\cE$. In this case, $b\smin$ can be
chosen to be $\eta$ and we have $\gamma_\cE=v(\zeta_p-1)$ and
\begin{equation}                       
I_{\cE}\>=\> (\zeta_p-1)\>.
\end{equation}
\sn
b) $\cE$ admits a Kummer generator $\eta$ which is a $1$-unit
with $v(\eta-1)\leq v(\zeta_p-1)$ and 
$1,\eta-1,\ldots,(\eta-1)^{p-1}$ a valuation basis for $\cE$. In this case,
$b\smin$ can be chosen to be $\eta-1$ and we have $\gamma_\cE=v(\zeta_p-1)-v(\eta-1)$
and
\begin{equation}                   \label{Icase2b}
I_{\cE}\>=\> \left(\frac{\zeta_p-1}{\eta-1}\right)\>.
\end{equation}
We have $I_{\cE}=\cO_L$ if and only if $v(\eta-1)=v(\zeta_p-1)$, and this holds if 
and only if $Lv|Kv$ is separable of degree $p$.
\end{theorem}
\begin{proof}
Throughout the proof we use part 1) of Corollary~\ref{part1&2}.
\sn
1): By part 1) of Proposition~\ref{ASKgen} there exists an Artin-Schreier generator 
$\vartheta$ of value $v\vartheta\leq 0$ such that $v\vartheta$ generates the value 
group extension, or 
$v\tilde{c}\vartheta=0$ and $Lv=Kv(\tilde{c}\vartheta v)$ for some $\tilde{c}\in K$.
By Lemma~\ref{vbdegp}, it follows that $1,\vartheta,\ldots,\vartheta^{p-1}$ is a 
valuation basis for $\cE$.

If $v\vartheta<0$, then 
\begin{equation}                      \label{IEASgen}
v\left(\frac{\sigma \vartheta}{\vartheta}-1\right)
\>=\>v\left(\frac{\sigma \vartheta-\vartheta}{\vartheta}\right)
\>=\>-v\vartheta\>=\>v\left(\frac{1}{\vartheta}\right)\> >\> 0\>
\end{equation}
for every $\sigma\in \Gal L|K\setminus \{\rmid\}$ since then $\sigma \vartheta-
\vartheta\in\F_p\setminus\{0\}$.
Hence by Lemma~\ref{lemv(sai/ai-1)}, for $1\leq j\leq p-1$ we have
\[
v\left(\frac{\sigma \vartheta^j}{\vartheta^j}-1\right)\>=\> 
v\left(\frac{\sigma \vartheta}{\vartheta}-1\right)
\>=\>v\left(\frac{1}{\vartheta}\right)\>.
\]
This proves that $b\smin$ can be chosen to be $\vartheta$ in this case.

\pars
If $v\vartheta=0$, which by part 1) of Proposition~\ref{ASKgen} holds if and only if
$Lv|Kv$ is separable of degree $p$, then 
\[
v\left(\frac{\sigma \vartheta}{\vartheta}-1\right)
\>=\>v\left(\frac{1}{\vartheta}\right)\> =\> 0\>,
\]
and as the value $\gamma_\cE$ is non-negative, this is
equivalent to $I_{\cE}=\cO_L\,$.

\mn
2): By part 2) of Proposition~\ref{ASKgen} there exists a Kummer generator $\eta$ such
that either
\n
a) $v\eta$ generates the value group extension, or $\eta v$ generates the residue field
extension, or
\n
b) $\eta$ is a $1$-unit and $v(\eta-1)$ generates the value group extension or for some 
$\tilde{c}\in K$, $\tilde{c}(\eta-1)v$ generates the residue field extension.
\pars
We first consider case a). By Lemma~\ref{vbdegp}, it follows that $1,\eta,\ldots,
\eta^{p-1}$ is a valuation basis for $(L|K,v)$. If $v\eta$ generates the value group
extension, then so does $v\eta^{-1}$. Therefore, we can assume that $v\eta\geq 0$. For
$1\leq j\leq p-1$ and every $\sigma\in \Gal L|K\setminus \{\rmid\}$,
\[
v\left(\frac{\sigma \eta^j}{\eta^j}-1\right)\>=\> 
v\left(\frac{\sigma \eta^j-\eta^j}{\eta^j}\right)\>=\>
v\left(\frac{\zeta_p^k \eta^j-\eta^j}{\eta^j}\right)\>=\>
v(\zeta_p^k-1)\>=\>v(\zeta_p-1)
\]
for some $k\in\N$; the last equation holds since $v(\zeta-1)=vp/(p-1)$ for every 
primitive $p$-th root of unity $\zeta$ (cf.\ \cite[Lemma~2.5]{pr1}). This proves that
in case a), $b\smin$ can be chosen to be $\eta$ and we have $\gamma_\cE=v(\zeta_p-1)$.

\parm
Now we consider case b). Again by Lemma~\ref{vbdegp}, $1,\eta-1,\ldots,(\eta-1)^{p-1}$ 
is a valuation basis for $(L|K,v)$. Since $v\eta=0$, we have for every $\sigma\in
\Gal L|K\setminus \{\rmid\}$:
\[
v\left(\frac{\sigma (\eta-1)}{\eta-1}-1\right)\>=\>
v\left(\frac{\sigma \eta-\eta}{\eta-1}\right)\>=\>
v(\zeta_p-1)-v(\eta-1)\>.
\]
This value must be non-negative since it is not less than $\gamma_\cE\,$. If it is equal
to $0$, then it must be equal to $\gamma_\cE\,$. If it is positive, 
then we can apply Lemma~\ref{lemv(sai/ai-1)}, obtaining that for $1\leq j\leq p-1$,
\[
v\left(\frac{\sigma (\eta-1)^j}{(\eta-1)^j}-1\right)\>=\>
v\left(\frac{\sigma (\eta-1)}{\eta-1}-1\right)\>
\]
and consequently, this value is again equal to $\gamma_\cE\,$.
Hence in case b), $b\smin$ can be chosen to be $\eta-1$ and we have $\gamma_\cE=
v(\zeta_p-1)-v(\eta-1)$.

Since $I_\cE$ is the unique ramification ideal of $\cE$, the equality $I_{\cE}=\cO_L$ is
equivalent to $L$ being the ramification field of $(L|K,v)$. This in turn holds if and
only if $(L|K,v)$ is a tame extension (\cite[Proposition 4.1]{Ku2}). Since $(L|K,v)$ is
assumed to be unibranched and defectless
of degree $p$, this is equivalent to $Lv|Kv$ being separable of degree $p$.
\end{proof}

\begin{remark}            \label{rem}
\rm Equation~(\ref{Icase2b}) also holds in case 2 a) of the previous theorem since in 
this case, $v(\eta-1)=0$. Indeed, in that case we have $v\eta\geq 0$, and $1,\eta,
\ldots,\eta^{p-1}$ form a valuation basis for $(L|K,v)$. If $v\eta>0$, then 
$v(\eta-1)=0$. If $v\eta=0$, then $1,\eta v,\ldots,(\eta v)^{p-1}$ form a basis of 
$Lv|Kv$, so $\eta v\ne 1$, whence $v(\eta-1)=0$ again.
\end{remark}

\mn
%
%
%
\subsubsection{The defect case}                                 \label{sectdefc}
\mbox{ }\sn
\begin{proposition}                      \label{S_EASKum}
If $\cE=(L|K,v)$ is an Artin-Schreier extension with Artin-Schreier generator
$\vartheta$ such that $v(\vartheta-K)$ has no maximal element, then $v(\sigma\vartheta
-\vartheta)=0$ for all $\sigma\in G\smid$, and
\begin{equation}                            \label{SigmaAS}
\Sigma_{\cE} \>=\> -v(\vartheta-K)\>.
\end{equation}

If $\cE$ is a Kummer extension with Kummer generator $\eta$ of value $0$ such that
$v(\eta-K)$ has no maximal element, then $v(\sigma \eta-\eta)=v(\zeta_p-1)$ for
all $\sigma\in G\smid$, where $\zeta_p$ is an arbitrary primitive $p$-th root of unity,
and
\begin{equation}                            \label{SigmaKum}
\Sigma_{\cE} \>=\> v(\zeta_p-1)\,-\,v(\eta-K) \>=\> \frac{vp}{p-1}   
\,-\,v(\eta-K)\>.
\end{equation}
\end{proposition}
\begin{proof}
Every Artin-Schreier generator $\vartheta$ satisfies $\sigma\vartheta-\vartheta
\in\F_p$ and hence $v(\sigma\vartheta-\vartheta)=0$ for all $\sigma\in G\smid$.
Every Kummer generator $\eta$ satisfies $\sigma\eta-\eta=(\zeta_p-1)\eta$ for some
primitive $p$-th root of unity, and as $v(\zeta_p-1)$ is the same for all primitive
$p$-th roots of unity $\zeta_p$, we have $v(\sigma \eta-\eta)=v(\zeta_p-1)+v\eta
=v(\zeta_p-1)$ for all
$\sigma\in G\smid$, for $\zeta_p$ an arbitrary primitive $p$-th root of unity.

Now equations (\ref{SigmaAS}) and (\ref{SigmaKum}) follow from equation
(\ref{ram_gp'}).
\end{proof}

From this proposition together with Corollary~\ref{SigmaEcor}, we obtain:
\begin{corollary}                           \label{ASdef}
1) An Artin-Schreier extension with Artin-Schreier generator $\vartheta$ is a defect
extension if and only if $v(\vartheta-K)$ has no maximal element and is contained in
$vL^{<0}$.
\sn
2) A Kummer extension with Kummer generator $\eta$ of value $0$ is a defect
extension if and only if $v(\eta-K)$ has no maximal element and is bounded from above by
$v(\zeta_p-1)$.
\sn
3) Let $(L|K,v)$ be an Artin-Schreier extension or a Kummer extension of prime degree.
Then $\cE$ is a defect extension if and only if $\Sigma_\cE$ is a subset of $vL^{>0}$
without a minimal element and therefore, $I_\cE$ is a nonprincipal $\cO_L$-ideal.
\end{corollary}

\begin{theorem}                                  \label{ThmIEdef}
Take a Galois defect extension $\cE=(L|K,v)$ of prime degree $p$. 
\sn
1) If $\cE$ is an Artin-Schreier extension with Artin-Schreier generator
$\vartheta$, then
\begin{eqnarray*}
I_{\cE}&=& \left(\frac{1}{\vartheta-c}\>\left|\;\> c\in K\right. \right)\\ 
&=& \left(\frac{1}{b}\>\left|\;\>b\, \mbox{ an Artin-Schreier generator of } 
L|K \right.\right)\>.
\end{eqnarray*}

\sn
2) Let $\cE$ be a Kummer extension with a Kummer generator $\eta$ which is a
$1$-unit, and $\zeta_p$ a primitive $p$-th root of unity. Then
\begin{eqnarray*}
I_{\cE}&=& \left(\frac{\zeta_p-1}{\eta-c}\>\left|\;\> c\in K \mbox{ a $1$-unit}\right. 
\right)\\
&=& \left(\frac{\zeta_p-1}{b-1}\>\left|\;\> b\, \mbox{ a Kummer generator of $L|K$ 
which is a $1$-unit}\right. \right)\>.
\end{eqnarray*}
\end{theorem}
\begin{proof}
1): The first equation follows from equation (\ref{IEdef}) of Corollary~\ref{SigmaEcor}, 
where we take $\sigma_0$ such that $\sigma_0 \vartheta=\vartheta+1$. The ideal on the right
hand side of the second equation contains the ideal on the right hand side of the first 
equation because $\vartheta-c$ is again an Artin-Schreier generator for every $c\in K$.
Further, by Corollary~\ref{SigmaEcor} the ideal on the right hand side of the second 
equation is contained in $I_\cE\,$. Hence the two ideals are equal.
\sn
2): The first equation follows from equation (\ref{IEdef}) of Corollary~\ref{SigmaEcor},
where we take $\sigma_0$ such that $\sigma_0\eta=\zeta_p\eta$, because then 
$\sigma_0(\eta-c)-(\eta-c)=(\zeta_p-1)\eta$ and we can drop $\eta$ since it is a unit. 
Further, we can restrict $c$
to $1$-units since if $c$ is not a $1$-unit, then $v(\eta-c)\leq 0 <v(\eta-1)$ and 
$\frac{\zeta_p-1}{\eta-c}\in\left(\frac{\zeta_p-1}{\eta-1}\right)$.

When $c$ is a $1$-unit, then $\eta-c=c(\frac \eta c -1)$, the quotient $b=\frac \eta c$ 
is again a Kummer generator which is a $1$-unit, and we can drop the unit factor $c$. 
This shows that the ideal on the right hand side of the second equation contains the ideal 
on the right hand side of the first equation. Further, by Corollary~\ref{SigmaEcor} the 
ideal on the right hand side of the second equation is contained in $I_\cE\,$. Hence the
two ideals are equal.
\end{proof}

\mn
%
%
%
\subsection{When does the equality $I_{\cE}=\cM_L$ hold?}      \label{sectIeqM}
\mbox{ }\sn
Throughout, we assume that $\cE=(L|K,v)$ is a purely wild Galois extension of degree
$p=\chara Kv$. If $\chara K =0$, we assume in addition that $K$ contains a primitive
$p$-th root of unity $\zeta_p\,$, so that $L|K$ is a Kummer extension. Under these
assumptions, we will determine the cases where $I_{\cE}=\cM_L\,$.

Let us discuss the meaning of the equality $I_{\cE}=\cM_L\,$. If $\cE$ is defectless,
then $I_{\cE}$ is a principal $\cO_L$-ideal, hence so is $\cM_L\,$, that is, both $vL$
and $vK$ have a minimal positive element. For a defect extension $\cE$ we will show
in Section~\ref{sectdefceq} that the equality implies that $\cE$ has independent defect in
the sense of \cite{KuRz} (see also \cite[Section 1.2]{Kudefc}). However, the converse
does only hold if $vL$ is archimedean, because otherwise $\cE$ has independent defect if
and only if $I_{\cE}$ is equal to the maximal ideal of some coarsening of $\cO_L$ whose
maximal ideal is nonprincipal. In any case, a defect extension $\cE$ satisfying $I_{\cE}=
\cM_L$ does not yield a definable proper coarsening of $\cO_L$ in the sense of
\cite{Kudefc}.

\sn
%
%
%
\subsubsection{The defectless case}                         \label{sectdefleq}
\mbox{ }\sn
We assume $\cE$ to be defectless. Then we know from Proposition~\ref{dlrip}
that $I_{\cE}$ is principal. Hence
for it to be equal to $\cM_L\,$, the latter and also $\cM_K$ must be principal. As
$\cE$ is purely wild, the extension $Lv|Kv$ cannot be separable of degree $p$.
\begin{proposition}
Let the assumptions be as described above.
\sn
1) Assume that $\cE$ is an Artin-Schreier extension. If $(vL:vK)=p$, then
$I_{\cE}=\cM_L$ if and only if $\cE$ admits an Artin-Schreier generator $\vartheta$
such that $-v\vartheta$ is the smallest positive element of $vL$ and $-pv\vartheta$ is
the smallest positive element of $vK$. If $[Lv:Kv]=p$, then
$I_{\cE}=\cM_L$ if and only if $\cE$ admits an Artin-Schreier generator $\vartheta$
such that $-v\vartheta$ is the smallest positive element of $vL=vK$.
\sn
2) Assume that $\cE$ is a Kummer extension. Then $I_{\cE}=\cM_L$ holds if and only if
$\cE$ admits a Kummer generator $\eta$ such that one of the following cases holds:
\sn
(a) $\rme(L|K,v)=p$ and $v\eta$ generates the value group extension or $\rmf(L|K,v)=p$,
$v\eta=0$ and $\eta v$ generates the residue field extension, and $\cM_L=(\zeta_p-1)
\cO_L$ as well as $\cM_K=(\zeta_p-1)\cO_K\,$,
\sn
(b) $\rme(L|K,v)=p$, $\eta$ a $1$-unit, $v(\eta-1)$ generates the value group extension
or $\rmf(L|K,v)=p$ and $v\tilde{c}(\eta-1)$ generates the residue field extension for
some $\tilde{c}\in K$, and $\cM_L=\frac{\zeta_p-1}{\eta-1}\cO_L\,$.
\end{proposition}
\begin{proof}
Our statements follow from Proposition~\ref{ASKgen} together with Theorem~\ref{ThmASKgen}.
In case 2(a) note that $\cM_L=(\zeta_p-1)\cO_L$ means that $v(\zeta_p-1)\in vK$ is the
smallest positive element in $vL$, which implies that it also is the smallest positive
element in $vK$.
\end{proof}

Note that if case 2(a) holds with $\rme(L|K,v)=p$, then $vK$ cannot be archimedean,
because $v(\zeta_p-1)$ is the smallest positive element of both $vL$ and $vK$.
\pars
Let us give examples for the different types of extensions appearing in the proposition.
\sn
$\bullet$ Artin-Schreier extension with $(vL:vK)=p$: take a valued field $(K,v)$ of
characteristic $p>0$ such that $vK$ has a smallest positive element $vc$, $c\in K$.
Let $\vartheta$ be a root of $X^p-X-c^{-1}$. Then $v\vartheta=-vc/p$.
Thus $v\vartheta^{-1}$ is the smallest positive element of
$vK(\vartheta)$, whence $I_\cE=\cM_L$ for $L=K(\vartheta)$.
\sn
$\bullet$ Artin-Schreier extension with $[Lv:Kv]=p$: take $(K,v)$ and $c\in K$ as
before. Further, assume that $Kv$ contains an element $dv$, $d\in \cO_K^\times$,
which does not have a $p$-th root in $Kv$. Let $\vartheta$ be a root of $X^p-X-c^{-p}d$.
Then $v\vartheta=-vc$ and $v((c\vartheta)^p-d)>0$, so that $c\vartheta v$ is a
$p$-th root of $dv$. We obtain $[Lv:Kv]=p$  for
$L=K(\vartheta)$, so $vL=vK$. As $v\vartheta^{-1}=vc$ is the smallest positive element
of $vK=vL$, it follows that $I_\cE=\cM_L\,$.
\sn
$\bullet$ Kummer extension with Kummer generator $\eta$ such that $v\eta$ generates
the value group extension, with $(vL:vK)=p$: take $K=\Q_p(\zeta_p,t)$, where $t$ is
transcendental over $\Q_p(\zeta_p)$ and extend the $p$-adic valuation to a valuation $v$
of $K$ in such a way that $vK$ is the lexicographic product $vt\,\Z\times v(\zeta_p-1)
\Z$. Let $\eta$ be a root of $X^p-t$ and set $L:=K(\eta)$. Then $(vL:vK)=p$, $v\eta$
generates the value group extension, and $vL$ is the lexicographic product $\frac{vt}{p}
\Z\times v(\zeta_p-1)\Z$. Consequently, $v(\zeta_p-1)$ is still the smallest positive
element of $vL$, showing that $I_\cE=\cM_L\,$.
\sn
$\bullet$ Kummer extension with Kummer generator $\eta$ such that $v\eta=0$ and
$\eta v$ generates the residue field extension, with $[Lv:Kv]=p$: take again $K=
\Q_p(\zeta_p,t)$, but now extend the $p$-adic valuation to a valuation $v$
of $K$ in such a way that $v$ is the Gauß valuation of the rational function field
$K=\Q_p(\zeta_p)(t)$. Then $tv$ is transcendental over $\Q_p(\zeta_p)v=\F_p$ and
does not have a $p$-th root in $Kv$. Let $\eta$ be a root of $X^p-t$ and set $L:=
K(\eta)$. Then $\eta v=(tv)^{1/p}$ and $\eta v$ generates the residue field extension.
Since $[Lv:Kv]=p$ implies $vL=vK$, $v(\zeta_p-1)$ is still the smallest positive
element of $vL$, showing that $I_\cE=\cM_L\,$.

\sn
$\bullet$ To construct Kummer extensions described in case 2(b), take a valued field
$(K,v)$ of characteristic $0$ with residue characteristic $p>0$ and assume that
$\zeta_p\in K$. Further, assume that $c\in K$ is a 1-unit such that $v(c-1)=vp$ and
that $a\notin K$ is a root of the polynomial
\begin{equation}                          \label{f}
f(X)\>=\>(X+1)^p-c\>=\> X^p+\sum_{0<i<p} \binom{p}{i} X^{p-i} +1-c\>.
\end{equation}
Then $\eta:=a+1$ is 1-unit and a Kummer generator of $K(a)|K$.

For $0<i<p$, each binomial coefficient in (\ref{f}) has value $vp>0$ and we have
$v\binom{p}{i} a^{p-i}\geq vp+va>vp$. Since $f(a)=0$ and $v(c-1)=vp$, it follows that
$va^p=v(c-1)$ and $v(a^p-(c-1))\geq vp+va>vp=va^p$.

\pars
Let us give an example of an extension with $(vL:vK)=p$.
Take $K=\Q(\zeta_p)$ with the extension $v$ of the $p$-adic valuation. Then $vK=
\frac 1 {p-1}\Z$, so $vp$ is not $p$-divisible in $vK$. Set $c=p+1$ so that $c-1=p$.
From our above computations we obtain that $va=\frac{vp}{p}$, which generates the value
group extension. Further,
\begin{equation}                       \label{spe}
v\,\frac{\zeta_p-1}{\eta-1}\>=\> \frac{vp}{p-1}-\frac{vp}{p}\>=\>\frac{vp}{p(p-1)}\>,
\end{equation}
which is the smallest positive element in $vL$. In view of part 2) b) of
Theorem~\ref{ThmASKgen}, this shows that $I_\cE=\cM_L\,$.

\pars
Now we give an example of an extension with $[Lv:Kv]=p$. Choose $\tilde{c}$ such that
$\tilde{c}^{-p}=p$ and extend the $p$-adic valuation to $\Q(\zeta_p,\tilde{c})$. Further,
take a transcendental element $t$, and $v$ to be the Gau{\ss} valuation on the rational
function field $K:=\Q(\zeta_p,\tilde{c})(t)$. Then $tv$ is transcendental over
$\Q(\zeta_p,\tilde{c})v=\F_p$ and does not have a $p$-th root in $Kv$. Set $c=1+pt$ so
that $c-1=pt$ and $v(c-1)=vpt=vp$. Now we obtain from our above computations that
$v(a^p-\tilde{c}^{-p}t)=v(a^p-(c-1))>va^p$. This
implies
\[
v(\tilde{c}^p (\eta-1)^p-t)\> =\> v(\tilde{c}^p a^p-t)\> >\>v\tilde{c}^p a^p \>=\>0\>,
\]
whence $\tilde{c}(\eta-1)v=(tv)^{1/p}$. We set $L:=K(\eta)=K(a)$ and observe that
$[Lv:Kv]=p$, so that $vL=vK$. Further, (\ref{spe}) again holds, so that
$v\frac{\zeta_p-1}{\eta-1}$ is the smallest positive element in $vK=vL$. Again by
part 2) b) of Theorem~\ref{ThmASKgen}, it follows that $I_\cE=\cM_L\,$.

\mn
%
%
%
\subsubsection{The defect case}                                 \label{sectdefceq}
\mbox{ }\sn
We assume $\cE$ to be an Artin-Schreier or a Kummer defect extension. Then part 3) of
Corollary~\ref{ASdef} shows that $vI_\cE=\Sigma_\cE$ has no
minimal element, so $I_{\cE}$ is nonprincipal. Hence for it to be equal to $\cM_L\,$,
both $\cM_L$ and $\cM_K$ must be nonprincipal.
We note:
\begin{lemma}
Let the assumptions be as described above. Then $I_{\cE}=\cM_L$ holds if and
only if $\,\Sigma_\cE=vL^{>0}$.
\end{lemma}

\parm
For a famous example of an Artin-Schreier extension with independent defect satisfying
the equation $I_{\cE}=\cM_L$ originally due to Shreeram Abhyankar, see
\cite[Example 3.12]{Ku31}.
For other examples, see \cite[Proposition 6.14]{Nov}. In that paper, also Galois
extensions $(L|K,v)$ with higher powers of $p$ as their degrees are constructed whose
unique ramification ideal is equal to $\cM_L\,$: see Corollary 6.10 and the
interesting system of Artin-Schreier extensions in Section 6.3.

\pars
Finally, we present an example of a Kummer extension with independent defect satisfying
the equation $I_{\cE}=\cM_L\,$; it has been developed with the help of Konstantinos
Kartas. We work over $\Q_p$ and assume $p\neq 2$ (for simplicity). Let $K$ be the
$p$-cyclotomic field which is obtained by adjoining a primitive $p$-th root of unity
$\zeta_p$ to $\Q_p$ and then closing under a compatible system of $p^n$-th roots of
$\zeta_p\,$. The $p$-adic valuation of $\Q_p$ extends (uniquely) to a valuation $v$ on $K$, which again is henselian.

We consider the unibranched extension $\cE=(K(p^{1/p})|K,v)$.
First, we note that $p^{1/p} \notin K$; otherwise $\Q_p(p^{1/p})$ would be
a subfield of the abelian extension $K|\Q_p$ and hence would be Galois over $\Q_p$
(since every subgroup of an abelian group is normal), which is not the case. Therefore,
$K(p^{1/p})|K$ is a proper extension and thus must be of degree $p$. As $\zeta_p\in K$,
it is a Kummer extension. We claim that $(K(p^{1/p})|K,v)$ is immediate. Since $vK$ is
$p$-divisible, it suffices to show that $K(p^{1/p})v=Kv$. But if this
were not true, then $K(p^{1/p})$ would be contained in the maximal unramified extension
of $K$, which is equal to $\Q_p^{ab}$. Again, this is a contradiction because this
would imply that $\Q_p(p^{1/p})|\Q_p$ is Galois. Therefore $(K(p^{1/p})|K,v)$ must be
immediate. As the extension is unibranched, it has defect $p$.

It is known that $(K,v)$ is a deeply ramified field (this follows e.g.\ from
computations in \cite{Ka2}). Hence by \cite[part (1) of Theorem 1.10]{KuRz},
$(K(p^{1/p})|K,v)$ has independent defect. We set $L:=K(p^{1/p})$; as an algebraic
extension of $\Q_p$, it has archimedean value group. The implication a)$\Rightarrow$b)
of \cite[Theorem~1.4]{pr1} shows that $I_\cE$ is the maximal ideal of some nontrivial
coarsening of $v$. But as the value group is archimedean, this can only be $\cM_L\,$,
so $I_{\cE}=\cM_L\,$.

\mn
%
%
%
\subsection{Defect does not always imply nonprincipality of ramification ideals.}                                 \label{sect2ex}
\n 
We are going to give an example of a Galois defect extension $(L|K,v)$ of degree $p^2$
which is a tower of two Galois extensions $L|L_0$ and $L_0|K$ of degree
$p=\chara K$ and rank one, the upper one defectless and the lower a defect extension,
but has only one ramification ideal, this being principal. We need a criterion for $L|K$
to be Galois. We set $\wp(X):=X^p-X$. The following is Lemma~2.9 in \cite{NN}:
\begin{lemma}                          \label{AStower}
Take Artin-Schreier extensions $L|L_0$ and $L_0|K$, and an Artin-Schreier generator
$\vartheta$ of $L|L_0$ with $\vartheta^p-\vartheta=b\in L_0$. Then $L|K$ is a Galois
extension if and only if $\sigma_0 b-b\in \wp(L_0)$ for some generator
$\sigma_0$ of $\Gal L_0|K$.
\end{lemma}

In order to construct a suitable field $K$ from which we can obtain the Artin-Schreier
defect extension $(L_0|K,v)$ we need, we choose a prime $q$ different from $p$ and use
\cite[Theorem~4.7]{B}. According to this theorem there exists a valued rational function
field $K=\widetilde{\F_p}(x,y)$ with value group $\Gamma:=\frac 1 {q^\infty}\Q$ such that
$(K(y^{1/p})|K,v)$ is immediate with $v(y^{1/p}-K)$ having supremum $-1/p$.
According to the proof we can define
\[
y\>:=\>\sum_{i\in\N} x^{p\gamma_i}\>\in\> \widetilde{\F_p}((x^{\Gamma}))\>,
\]
where the increasing sequence $\gamma_i=\frac {n_i}{q^i}$ with suitable $n_i\in\Z$
converges to $-1/p$. We are free to choose $\gamma_1=-1$.

Let $\vartheta_0$ be a root of the Artin-Schreier polynomial $X^p-X-y$ and set $L_0:=
K(\vartheta_0)$. We compute:
\[
v(\vartheta_0-y^{1/p})\>=\>\frac 1 p (\vartheta_0^p-y)\>=\> \frac 1 p v\vartheta_0
\>=\> \frac 1 p vy^{1/p}\>=\>-1/p\>.
\]
From \cite[part 4) of Lemma~2.9]{Ku65} it follows that $v(\vartheta_0-K)=v(y^{1/p}-K)
\subseteq vL_0^{<0}$. Hence the
Artin-Schreier extension $\cE_0:=(L_0|K,v)$ satisfies $\Sigma_\cE=-v(\vartheta_0-K)
\subseteq vL_0^{>0}$ and is therefore an immediate defect extension by
Corollary~\ref{SigmaEcor}. We have that $I_{\cE_0}$ is nonprincipal, with $vI_{\cE_0}$
having infimum $\gamma:=1/p$. However, we will construct the extension $(L|K,v)$ such
that $I_{\cE_0}$ does not induce a nonprincipal ramification ideal of it.

\pars
Let $\vartheta$ be a root of the Artin-Schreier polynomial $X^p-X-\vartheta_0\,$, and
set $L:=L_0(\vartheta)=K(\vartheta_0,\vartheta)$. Since $v\vartheta_0=-vy/p=-1<0$, we
have $v\vartheta=-1/p\notin vK=vL_0\,$. Hence by Corollary~\ref{corapprdle}, the elements
$1,\vartheta,\ldots,\vartheta^{p-1}$ form a valuation basis of $\cE_1:=(L|K(\vartheta_0)
,v)$, showing that this extension is defectless. By part 1) of Theorem~\ref{ThmASKgen},
\[
I_{\cE_1} \>=\> \left(\frac{1}{\vartheta} \right)\>,
\]
so the minimum of $vI_{\cE_1}$ is $-v\vartheta=1/p=\gamma$.

Since $\vartheta^p-\vartheta=\vartheta_0\,$, we have $L=K(\vartheta)$. To show that
$L|K$ is a Galois extension, take some generator
$\sigma_0$ of $\Gal L_0|K$. Since $\sigma_0\vartheta_0$ is also
a root of $X^p-X-y$, we have $\sigma_0\vartheta_0-\vartheta_0=i$ for some
$i\in\F_p\,$. As $K$ contains $\widetilde{\F_p}$, it contains the Artin-Schreier roots
of $i$, i.e., $i\in\wp(K)\subseteq\wp(L_0)$. Now Lemma~\ref{AStower} shows that
$L|K$ is a Galois extension. We set $G:=\Gal L|K$. At this point, we can apply
Proposition~\ref{compri} to find that (\ref{gamma=minL}) holds.

By Corollary~2.10 of \cite{NN}, $G$ it is not cyclic,
and the discussion leading up to this corollary shows the following. Take $\sigma\in G=
\Gal L|K$ such that $\sigma \vartheta_0-\vartheta_0=1$. Then $\xi:=\sigma\vartheta-
\vartheta$ satisfies
\begin{eqnarray}
\xi^p-\xi &=& (\sigma\vartheta-\vartheta)^p-(\sigma\vartheta-\vartheta)\>=\>
(\sigma\vartheta)^p-\vartheta^p-\sigma\vartheta+\vartheta\\
&=&\sigma(\vartheta^p-\vartheta)
-(\vartheta^p-\vartheta)\>=\>\sigma\vartheta_0-\vartheta_0\>=\>1
\end{eqnarray}
and is therefore an element of $\widetilde{\F_p}\subset K_0\,$. Hence $\sigma\xi=\xi$, and
by induction on $n\in\N$ one shows that $\sigma^i\vartheta =i\xi+\vartheta$ for every
$i\in\N$.

Further, take $\tau\in\Gal L|L_0\subset G$ such that $\tau\vartheta-\vartheta=1$. Then
$\tau$ is trivial on $L_0$ and $\sigma$ and $\tau$ commute. Thus every element in $G$ can
be written as $\sigma^i\tau^j$, $0\leq i,j\leq p-1$.

By induction on $j\in\N$ one shows that $\tau^j\vartheta=\vartheta+j$ for every
$n\in\N$. So we have
\[
\sigma^i\tau^j \vartheta-\vartheta\>=\> \sigma^i(\vartheta+j) -\vartheta\>=\>
i\xi+\vartheta+j-\vartheta\>=\>i\xi+j\,\in\,\widetilde{\F_p}\>,
\]
hence $v(\sigma^i\tau^j \vartheta-\vartheta)=0$ and
\begin{equation}                        \label{mingamma}
v\left(\frac{\sigma^i\tau^j \vartheta}{\vartheta} - 1\right)\>=\>-v\vartheta\>=\>
\gamma\>.
\end{equation}

Take any nontrivial subgroup $H$ of $G$. Then (\ref{gamma=minL}) implies that
\[
\min\left\{\left. v\left(\frac{\sigma b}b -1\right)\,\right|\, \sigma\in H\,,\,b\in
L^\times \right\}\>\geq\>\gamma\>.
\]
On the other hand, $H$ must contain some $\sigma^i\tau^j$ with $0\leq i,j\leq p-1$, so
(\ref{mingamma}) shows that (\ref{gamma=minL}) actually holds with $H$ in place of $G$.
Hence $I_H=I_G=I_{\cE_1}$ for all nontrivial subgroups $H$ of $G$. We have now proved:
\begin{proposition}                 \label{counterexAS}
There are Galois extensions of degree $p^2$ of valued fields of rank one in equal
characteristic $p>0$ that have only one ramification ideal, and this ramification ideal is
principal although the extension is not defectless.
\end{proposition}

The connection between the number of ramification ideals in a finite Galois extension and
its depth is studied in \cite{Nov,Nov2}. For the notion of depth, see \cite{NN}.

\bn\bn\bn


\begin{thebibliography}{99}

\bibitem{A} Abhyankar, S.: {\it Local uniformization on algebraic surfaces over ground
fields of characteristic $p \ne 0$}, Ann.\ of Math.\ {\bf 63} (1956) 491--526

\bibitem{B} Blaszczok, A.: {\it Distances of elements in valued field extensions},
Manuscripta Math.\ {\bf 159} (2019), 397--429.

\bibitem{[BK2]} Blaszczok, A.\ -- Kuhlmann, F.-V.: {\it On maximal immediate 
extensions of valued fields}, Mathematische Nachrichten {\bf 290} (2017), 7--18

\bibitem{CP} Cossart, V.\ -- Piltant, O.: {\it Resolution of singularities of threefolds
in positive characteristic. I. Reduction to local uniformization on Artin-Schreier and
purely inseparable coverings}, J.\ Algebra {\bf 320} (2008), 1051--1082

\bibitem{Ku58}
Cubides Kovacsics, P. - Kuhlmann, F.-V. - Rzepka (formerly Blaszczok), A.:
{\it On valuation independence and defectless extensions of valued fields}, J.\
Algebra {\bf 555} (2020), 69--95

\bibitem{pr1}
Cutkosky, S.D.\ -- Kuhlmann, F.-V.\ -- Rzepka, A.: {\it On the computation of K\"ahler
differentials and characterizations of Galois extensions with independent defect},
Math.\ Nachrichten {\bf 298} (2025), 1549--1577

\bibitem{pr2}
Cutkosky, S.D.\ -- Kuhlmann, F.-V.: {\it K\"ahler differentials of extensions of valuation
rings and deeply ramified fields}, submitted; available at
https://fvkuhlmann.de/Fvkprepr.html

\bibitem{DN} Dutta, A.\ -- Novacoski, J.: {\it Purity and distances between conjugates
of elements over a henselian valued field}, arXiv:2509.10839

\bibitem{En} Endler, O.: {\it Valuation theory}, Springer-Verlag, Berlin, 1972

\bibitem{EP} {Engler, A.J.\ -- Prestel, A.: {\it Valued fields},
Springer Monographs in Mathematics. Springer-Verlag, Berlin, 2005}


\bibitem{K} Kaplansky, I.: {\it Maximal fields with valuations I}, Duke Math.\ Journ.\
{\bf 9} (1942), 303--321

\bibitem{Ka2}
Kartas, K.: {\it Decidability via the tilting correspondence}, Algebra and
Number Theory {\bf 18} (2024), 209--24

\bibitem{Ku23}  
Knaf, H.\ -- Kuhlmann, F.-V.: {\it Abhyankar places admit local uniformization in any
characteristic}, Ann.\ Scient.\ Ec.\ Norm.\ Sup.\ {\bf 38} (2005), 833--846

\bibitem{Ku26}  
Knaf, H.\ -- Kuhlmann, F.-V.: {\it Every place admits local uniformization in a finite
extension of the function field}, Advances Math.\ {\bf 221} (2009), 428--453

\bibitem{Ku11}  
Kuhlmann, F.-V.: {\it Valuation theoretic and model theoretic aspects
of local uniformization}, in: Resolution of Singularities --- A Research Textbook in
Tribute to Oscar Zariski. H.~Hauser, J.~Lipman, F.~Oort, A.~Quiros (eds.), Progress in
Mathematics Vol.\ {\bf 181}, Birkh\"auser Verlag Basel (2000), 381--456

\bibitem{Ku29}  
Kuhlmann, F.-V.: {\it Elimination of Ramification I: The Generalized Stability Theorem},
Trans.\ Amer.\ Math.\ Soc.\ {\bf 362} (2010), 5697--5727

\bibitem{Ku30}  
Kuhlmann, F.-V.: {\it A classification of Artin--Schreier defect extensions and a
characterization of defectless fields}, Illinois J.\ Math.\ {\bf 54} (2010), 397--448

\bibitem{Ku31} Kuhlmann F.-V.: {\it The defect}, in: Commutative Algebra - Noetherian
and non-Noetherian perspectives, Fontana, M., Kabbaj, S.-E., Olberding, B., Swanson,
I. (Eds.), Springer-Verlag, New York, 2011

\bibitem{Ku34}  
Kuhlmann, F.-V.: {\it Approximation of elements in henselizations},
Manuscripta Math.\ {\bf 136} (2011), 461--474

\bibitem{Ku39}  
Kuhlmann, F.-V.: {\it The algebra and model theory of tame valued
fields}, J.\ reine angew.\ Math.\ {\bf 719} (2016), 1--43

\bibitem{Ku55}  
Kuhlmann, F.-V.: {\it Elimination of Ramification II: Henselian
Rationality}, Israel J.\ Math.\ {\bf 234} (2019), 927--958

\bibitem{Ku65} Kuhlmann, F.-V.: {\it Approximation types describing
extensions of valuations to rational function fields}, Mathematische Zeitschrift 
{\bf 301} (2022), 2509--2546

\bibitem{Kudefc} Kuhlmann, F.-V.: {\it On certain definable coarsenings of valuation rings
and their applications}, submitted; available at https://fvkuhlmann.de/Fvkprepr.html

\bibitem{Kucuts}
Kuhlmann, F.-V.\ -- Kuhlmann, K.: {\it Arithmetic of cuts in ordered abelian
groups and of ideals over valuation rings}, Pacific J.\ Math.\ {\bf 340} (2026), 339--374

\bibitem{Ku45}
Kuhlmann, F.-V.\ -- Pal, K.: {\it The model theory of separably tame fields}, J.\ Alg.\
{\bf 447} (2016), 74--108

\bibitem{Ku2}
Kuhlmann, F.-V.\ -- Pank, M.\ -- Roquette, P.: {\it Immediate and purely wild
extensions of valued fields}, Manuscripta math.\ {\bf 55} (1986), 39--67


\bibitem{KuRz}
Kuhlmann, F.-V.\ -- Rzepka, A.: {\it The valuation theory of deeply ramified fields and
its connection with defect extensions}, Transactions Amer.\ Math.\ Soc.\ {\bf 376}
(2023), 2693--2738

\bibitem{L} {Lang, S.: {\it Algebra}, revised 3rd.\ ed., Springer-Verlag, New York, 
2002}

\bibitem{M} Marshall, M.\ A.: {\it Ramification theory for valuations of arbitrary rank},
Canadian J.\ Math.\ {\bf 26} (1974), 908--916
 
\bibitem{NN} Nart, E.\ -- Novacoski, J.: {\it Depth of Artin-Schreier defect towers},
J.\ Pure Appl.\ Algebra {\bf 230} (2026) 108184

\bibitem{Nov} Novacoski, J.: {\it On the distances of an element to its conjugates},
arXiv:2505.15258

\bibitem{Nov2} Novacoski, J.: {\it Ramification ideals for products of pure subgroups},
arXiv:2608.21210

\bibitem{R} Ribenboim, P.: {\it Sur les groupes totalement ordonn\'es et
l'arithm\'etique des anneaux de valuation}, Summa Brasil.\ Math.\ {\bf 4} (1958), 1--64.

\bibitem{R0} {Ribenboim, P.: {\it Th\'eorie des valuations}, Les
Presses de l'Uni\-versit\'e de Mont\-r\'eal, Mont\-r\'eal, 2nd ed.\ (1968)}

\bibitem{R1} Ribenboim, P.: {\it Higher ramification groups for rank one valuations}, 
Math.\ Ann.\ {\bf 173} (1967), 253--259

\bibitem{R2} Ribenboim, P.: {\it Corrections to: ``Higher ramification groups for rank 
one valuations''}, Math.\ Ann.\ {\bf 185} (1970), 22--24

\bibitem{R3} Ribenboim, P.: {\it The theory of classical valuations}, Springer
Monographs in Mathematics, 1998

\bibitem{Th2} Thatte, V.: {\it Ramification theory for degree $p$ extensions of
valuation rings in mixed characteristic (0,p)}, J.\ Algebra {\bf 507} (2018), 225--248

\bibitem{W} Warner, S.: {\it Topological fields}, Mathematics studies {\bf 157},
North Holland, Amsterdam, 1989

\bibitem{ZS2} Zariski, O.\ -- Samuel, P.: {\it Commutative Algebra}, Vol.\ II,
New York--Heidelberg--Berlin, 1960

\end{thebibliography}
\end{document}